\documentclass[preprint,12pt,authoryear]{elsarticle}

\usepackage{amssymb}
\usepackage{amsthm}
\newtheorem{theorem}{Theorem}[section]
\newtheorem{lemma}[theorem]{Lemma}
\newtheorem{definition}{Definition}[section]
\newtheorem{problem}{Problem}[section]

\usepackage{amsmath}

\usepackage{tikz}
\usepackage{tikzscale}
\usetikzlibrary{external}
\newlength\figureheight
\newlength\figurewidth
\usepackage{pgfplots}

\usepackage{algorithm}
\usepackage{algpseudocode}

\usepackage{xcolor}

\journal{Applied Numerical Mathematics}

\begin{document}
	
	\begin{frontmatter}
		
		
		 \title{Algorithms for Modifying Recurrence Relations of Orthogonal Polynomial and Rational Functions when Changing the Discrete Inner Product}
		
		\author[b]{Marc Van Barel}
		\ead{marc.vanbarel@kuleuven.be}
		\affiliation[b]{organization={Department of Computer Science, KU Leuven, University of Leuven},
			addressline={Celestijnenlaan 200A}, 
			city={Leuven},
			postcode={3001}, 
			country={Belgium}}
		
		\author[a]{Niel Van Buggenhout\corref{cor1}}
		\ead{buggenhout@karlin.mff.cuni.cz}
		\cortext[cor1]{Corresponding author}
		\affiliation[a]{organization={Faculty of Mathematics and Physics, Charles University},
			addressline={Sokolovsk\'a 83}, 
			city={Praha},
			postcode={186 75}, 
			country={Czech Republic}}
		
		\author[b]{Raf Vandebril}
		\ead{raf.vandebril@kuleuven.be}
		

		\begin{abstract}
		    Often, polynomials or rational functions, orthogonal for a particular inner product are desired. In practical numerical algorithms these polynomials are not constructed, but instead the associated recurrence relations are computed.  Moreover, also typically the inner product is changed to a discrete inner product, which is the finite sum of weighted functions evaluated in specific nodes. For particular applications it is beneficial to have an efficient procedure to update the recurrence relations when adding or removing nodes from the inner product.
			The construction of the recurrence relations is equivalent to computing a structured matrix (polynomial) or pencil (rational) having prescribed spectral properties. Hence the solution of this problem is often referred to as solving an Inverse Eigenvalue Problem. In \cite{VBVBVa22} we proposed updating techniques to add nodes to the inner product while efficiently updating the recurrences. To complete this study we present in this article manners to efficiently downdate the recurrences when removing nodes from the inner product.
			The link between \textit{removing nodes} and the \textit{QR algorithm to deflate eigenvalues} is exploited to develop efficient algorithms. We will base ourselves on the perfect shift strategy and develop algorithms, both for the polynomial case and the rational function setting. Numerical experiments validate our approach.

		\end{abstract}
		
%
		
		
	\end{frontmatter}
	
	
	\section{Introduction}\label{intro}
	Given the nodes $z_i\in\mathbb{C}$ and weights $w_i\in\mathbb{C}$, consider the following  finite discrete inner product:
	\begin{equation}
		\langle f,g\rangle_m := \sum_{i=1}^{m} \vert w_i\vert^2 \overline{g(z_i)} f(z_i).
	\end{equation}
	In this paper we will consider a sequence of polynomials or rational functions
	with prescribed poles that is orthogonal with respect to this inner product.
	This sequence of orthogonal functions is characterized by a recurrence relation where the recurrence coefficients
	can be grouped together in a Hessenberg matrix for orthogonal polynomials, see \cite{b009}, or a Hessenberg pencil for 
	orthonormal rational functions, see \cite{VBVBVa22}.
	A method for the construction of a sequence of orthogonal polynomials or rational functions with respect to the positive semidefinite inner product
	is based on manipulating the corresponding recurrence matrix or recurrence pencil.
	A method to compute this Hessenberg matrix or Hessenberg pencil based on repeatedly adding a node until all nodes of the inner product are added, i.e., an updating problem, is equivalent to solving an inverse eigenvalue problem. 
	For orthonormal polynomials, when the nodes $z_i$ are real, \cite{m268} give such a method to compute a Jacobi matrix, i.e., a Hermitian Hessenberg matrix. This algorithm can be traced back to \cite{p468} and is an order of magnitude faster than for general nodes $z_i$. \cite{m267} uses this method to solve polynomial least squares problems by using the property that its solution can be found by projection onto a sequence of polynomials orthogonal to the given data.
	When the nodes are on the unit circle $z_i\in\mathbb{C}$, another efficient algorithm can be constructed resulting in a unitary Hessenberg instead of a Jacobi matrix, see \cite{m535,m256,p478}. For a survey of inverse eigenvalue problems, we refer to \cite{m887}.
	These methods were generalized to orthonormal polynomial vectors, see \cite{ma015,ma017,ma013,ma011,q331}.
	For orthogonal rational functions with prescribed poles, recently an updating procedure was proposed 
	based on the representation of the recurrence relation as a Hessenberg pencil, see \cite{VBVBVa22,ma054,p453,b141}.\\
	Instead of adding a node to the inner product, we can also remove a node, which is called a downdating problem.
	The downdating problem, due to its nature, is more difficult to perform in a numerically stable way, see \cite{n337}.
	For downdating orthogonal polynomials when the nodes $z_i$ are on the unit circle, we refer to \cite{m275}.
	A recent result by \cite{MaVD18} on deflating a known eigenvalue from a Hessenberg matrix suggests a reliable downdating procedure for orthogonal polynomials.
	We use their result to propose a novel downdating procedure for orthogonal polynomials and compare it to an alternative QR procedure, originally proposed for unitary Hessenberg matrices by \cite{m256}.
	For orthonormal rational functions we use the Hessenberg pencil representation and propose generalizations of the methods mentioned above.
	The generalization of the QR procedure to a Hessenberg pencil is the rational QZ method from \cite{CaMeVa19}.
	
	The paper is organised as follows.
	In Section~\ref{sec:problem} inverse eigenvalue problems as well as the downdating problem for orthogonal polynomial and rational functions with respect to a discrete inner product are defined.
	For the polynomial case, a Hessenberg matrix has to be constructed satisfying certain spectral properties while for the rational case, it is a Hessenberg pencil.
	In Section~\ref{sec:down:poly} three different methods are developed for solving the downdating problem for the polynomial case based on the theory given in Section~\ref{Hess:eig}.
	Based on the RQ algorithm for Hessenberg matrices an implicit as well as an explicit method are given (the implicit and explicit matrix method)
	The third method is based on the eigenvector corresponding to the eigenvalue (perfect shift) that has to be deflated (the eigenvector method).
	Section \ref{sec:numExp:poly} illustrates the three methods by several numerical experiments.
	To solve the downdating problem for orthogonal rational functions, two numerical methods are given in Section~\ref{sec:NumAlg:RF} based on the theory developed in Sections~\ref{sec:QRKrylov} and \ref{sec:GEVP}.
	Numerical experiments are described in Section~\ref{sec:NumExp:RF}.
	In Section~\ref{sec:sliding_window}, as an application a sliding window scheme to approximate a function is designed based on the updating and downdating procedure for orthogonal rational functions. This is illustrated by a numerical experiment.

	\section{Problem (re)formulation}\label{sec:problem}
	The problems of constructing polynomials and rational functions, orthogonal for a given discrete inner product, can be recast into matrix problems. In this article we will base all techniques on the associated matrix problem: we will examine how up- and downdating the nodes in the inner product relates to up- and downdating the matrix or pencil of recurrences.

	The inner product considered throughout the paper is a finite discrete inner product:
	\begin{equation}\label{eq:inprod}
		\langle f,g\rangle_m := \sum_{i=1}^{m} \vert w_i\vert^2 \overline{g(z_i)} f(z_i),
	\end{equation}
	with nodes $z_i\in\mathbb{C}$ and weights $w_i\in\mathbb{C}$.
	The restriction to such inner products is natural, since a continuous inner product can be approximated by an appropriate quadrature rule, resulting in a finite discrete inner product \cite[p.90]{Ga04book}. The selection of nodes and weights is of course problem specific and can change when the problem changes, hence the proposed algorithms for efficiently modifying the recurrences.

	\subsection{Orthogonal polynomials}
	We denote the space of polynomials by $\mathcal{P}$ and the space of polynomials up to degree $\ell$ by $\mathcal{P}_\ell$.
	A sequence of discrete orthonormal polynomials $\{p_\ell\}_{\ell=0}^{m-1}$ is characterized by a condition on the degree $p_\ell\in\mathcal{P}_{\ell}\backslash\mathcal{P}_{\ell-1}$ and the orthonormality condition $\langle p_i,p_j \rangle_m	= \delta_{i,j}$, with $\delta_{i,j}$ the Kronecker delta.\\
	Consider the following inverse eigenvalue problem (Problem~\ref{problem:HIEP}): given eigenvalues and partial eigenvector information, construct a structured matrix (Hessenberg form) having the desired eigenvalue and eigenvector information.
	\begin{problem}[Hessenberg Inverse Eigenvalue Problem (HIEP)]\label{problem:HIEP}
		Given $Z=\textrm{diag}(z_1,\dots,z_m)$, $w = \begin{bmatrix}
			w_1 & \dots & w_m
		\end{bmatrix}^\top$, construct $Q,H_m\in\mathbb{C}^{m\times m}$ such that
		\begin{align*}
			ZQ = QH_m, \quad Q^H Q = I, \quad	Q e_1 = w/\Vert w\Vert, \quad H_m \text{ has Hessenberg structure}.
		\end{align*}
	\end{problem}
	
	Given a discrete inner product \eqref{eq:inprod}, then the
	problem of constructing the associated orthogonal polynomials
	amounts to solving Problem~\ref{problem:HIEP}, with eigenvalues
	equal to the nodes and the $w_i$ as weights.  The matrix of
	recurrences will be the Hessenberg matrix $H_m$, see \cite{b009}. For
	particular choices of nodes, additional structure can be imposed
	on the Hessenberg matrix. For example, if $z_i\in\mathbb{R}$, the
	generated matrix is a Jacobi matrix; for unimodular complex
	nodes the Hessenberg matrix will be unitary, see \cite{m257,m268,p513}.
	In this article we deal with the generic \emph{unstructured} Hessenberg case.
	
	Recently \cite{VBVBVa22} proposed a technique to efficiently solve Problem~\ref{problem:HIEP} by efficiently adding the nodes incrementally, one after the other and updating the existing Hessenberg matrix of recurrences.	
	Downdating of a Hessenberg matrix is formulated in Problem \ref{problem:downdateHess} and essentially amounts to removing an eigenvalue from the Hessenberg's spectrum.
	\begin{problem}[Downdate HIEP]\label{problem:downdateHess}
		Consider the solution $H_m\in\mathbb{C}^{m\times m}$ to Problem \ref{problem:HIEP}, with nodes $Z = \textrm{diag}(z_1,\dots,z_m)$ and a vector containing the weights $w=\begin{bmatrix}
			w_1 & \dots & w_m
		\end{bmatrix}^\top$.
		Given a node $\tilde{z}\in \{z_i \}_{i=1}^m$, assume, without loss of generality, $\tilde{z} = z_j$.
		Denote by $\tilde{Z} = \textrm{diag}(\{z_i \}_{i=1,i\neq j}^m)$ the new matrix of nodes and let $\tilde{w} = \begin{bmatrix}
			w_1 & \dots w_{j-1} & w_{j+1} & \dots w_m
		\end{bmatrix}^\top$.
		Compute the Hessenberg matrix $\tilde{H}_{m-1}\in\mathbb{C}^{(m-1)\times (m-1)}$ such that
		\begin{align*}
			\tilde{Z} \tilde{Q} = \tilde{Q} \tilde{H}_{m-1}, \quad \tilde{Q}^H \tilde{Q} = I \quad \text{and}\quad \tilde{Q}e_1 = \tilde{w}/\Vert\tilde{w} \Vert_2.
		\end{align*}
	\end{problem}
	
	The solution to Problem \ref{problem:downdateHess}, i.e., the Hessenberg matrix $\tilde{H}_{m-1}$, is the recurrence matrix for a sequence of orthogonal polynomials $\{\tilde{p}_l \}_{l=1}^{m-1}$ satisfying the inner product
	\begin{equation*}
		\langle f,g\rangle_{\sim} := \sum_{i=1,i\neq j}^{m} \vert w_i\vert^2 \overline{g(z_i)} f(z_i).
	\end{equation*}
	
	Algorithms for solving Problems~\ref{problem:HIEP} and \ref{problem:downdateHess} allow us to efficiently modify the matrix of recurrences such that it matches the changed inner product.
	It is important to note that the Hessenberg matrices we are dealing with are normal. They are hence unitarily diagonalizable, and have only simple eigenvalues. This is no constraint at all, as the matrices we consider from the Inverse Eigenvalue Problem, satisfy these constraints naturally.
	
	\subsection{Orthogonal rational functions}
	A rational function $r(z)\in\mathcal{R}$ is the ratio of two polynomials, 
	\begin{equation*}
		r(z) = \frac{p(z)}{q(z)}, \quad p(z), q(z)\in\mathcal{P}.
	\end{equation*}
	Our interest is in rational functions with prescribed poles.
	Given a set of poles $\Xi = \{\xi_1,\xi_2,\dots, \xi_k\}$, with $\xi_i\in\overline{\mathbb{C}}$ ($\overline{\mathbb{C}}=\mathbb{C}\cup{\infty}$):
	\begin{equation}\label{eq:ratFct}
		r(z) = \frac{p(z)}{\pi(z)}, \quad p(z)\in\mathcal{P} \text{ and } \pi(z)=\underset{\xi_i \neq \infty }{\prod_{i=1}^{k}}(z-\xi_i).
	\end{equation}
	The space formed by these rational functions will be denoted by $\mathcal{R}^{\Xi}$, which is a short notation for $\mathcal{P}/\pi(z)$:
	The finite dimensional space $\mathcal{R}^{\Xi}_k$ is defined as $\textrm{span}\{1,r_1(z),\dots,r_k(z)\}$, with
	\begin{equation}\label{eq:ratFctPoles}
		r_{\ell}(z) = \frac{p_{\ell}(z)}{\pi_\ell(z;\Xi)}, \quad \pi_\ell(z;\Xi) := \underset{\xi_i \neq \infty }{\prod_{i=1}^{\ell}}(z-\xi_i), \quad p_{\ell} \in \mathcal{P}_\ell\backslash\mathcal{P}_{\ell-1}.
	\end{equation}
	The set of poles $\Xi$ for $\mathcal{R}^{\Xi}_k$ must contain (at least) $k$ poles $\xi_i\in\overline{\mathbb{C}}$.
	A sequence $\{r_\ell\}_{\ell=0}^{m-1}$ of discrete orthogonal rational functions (ORFs) is characterized by $r_\ell\in\mathcal{R}^{\Xi}_\ell\backslash\mathcal{R}^{\Xi}_{\ell-1}$ and $		\langle r_i,r_j\rangle_m = \delta_{i,j}$.
	Whereas in the polynomial setting the recurrences of the orthogonal polynomials are found in a Hessenberg matrix, here the recurrences for generating the orthogonal rational functions will be stored in a Hessenberg pencil. Again the eigenvalues of the pencil link to the nodes of the inner product and the weights link to information on the eigenvectors; the poles, however, are found as the ratio of the subdiagonal elements of the Hessenberg matrices in the pencil.
	Constructing the recurrence relations for discrete ORFs can be formulated in terms of an inverse eigenvalue problem for a Hessenberg pencil, see \cite{VBVBVa22}.
	\begin{problem}[Hessenberg Pencil Inverse Eigenvalue Problem (HPIEP)]\label{problem:HPIEP}
		Given $Z=\textrm{diag}(z_1,\dots,z_m)$, $w = \begin{bmatrix}
			w_1 & \dots & w_m
		\end{bmatrix}^\top$ and $\Xi = \{\xi_1,\dots,\xi_{m-1}\}$ with $\xi_i\in\overline{\mathbb{C}}\backslash \{z_j\}_{j=1}^m$, construct $Q,H_m,K_m\in\mathbb{C}^{m\times m}$, with $(H_m,K_m)$ a Hessenberg pencil, such that
		\begin{align*}
			ZQK_m = QH_m, \quad Q^H Q = I, \quad	Q e_1 = w/\Vert w\Vert,
		\end{align*}
		and $h_{i+1,i}/k_{i+1,i} = \xi_i$, where $h_{i,j}$ and
		$k_{i,j}$ denote the individual elements of the matrices $H_m$ and $K_m$.
	\end{problem}
	The downdating problem for ORFs is similar to downdating  OPs.
	A node-weight pair $(\tilde{z},\tilde{w})$ can be removed from the inner product \eqref{eq:inprod} as well as a freely chosen pole $\tilde{\xi}$.
	The pole $\tilde{\xi}$ can be any pole, that is, if $\tilde{z}=z_j$, then we are not obliged to choose $\tilde{\xi}=\xi_{j}$.
	This is because of the fact that the choice of space $\mathcal{R}^{\Xi}$ and the inner product $\langle .,.\rangle_m$ are not related.

	\begin{problem}[Downdate HPIEP]\label{problem:downdateHP}
		Consider $(H_m,K_m)\in\mathbb{C}^{m\times m}\times\mathbb{C}^{m\times m}$, the solution to Problem \ref{problem:HPIEP}, with nodes $Z = \textrm{diag}(z_1,\dots,z_m)$ and a vector containing the weights $w=\begin{bmatrix}
			w_1 & \dots & w_m
		\end{bmatrix}^\top$.
		Given a node $\tilde{z}\in \{z_i \}_{i=1}^m$, assume, without loss of generality, $\tilde{z} = z_j$ and a pole $\tilde{\xi}\in\Xi$, assume this pole is $\tilde{\xi} = \xi_\ell$.
		Denote by $\tilde{Z} = \textrm{diag}(\{z_i \}_{i=1,i\neq j}^m)$ the new matrix of nodes, the weights by $\tilde{w}=\begin{bmatrix}
			w_1 & \dots w_{j-1} & w_{j+1} & \dots w_m
		\end{bmatrix}^\top$ and by $\tilde{\Xi} = \{\xi_1,\dots,\xi_{\ell-1},\xi_{\ell+1},\dots,\xi_{m-1} \}$ the set of poles.
		Compute the Hessenberg pencil $(\tilde{H}_{m-1},\tilde{K}_{m-1})\in\mathbb{C}^{(m-1)\times (m-1)}\times \mathbb{C}^{(m-1)\times (m-1)}$ such that
		\begin{align*}
			\tilde{Z} \tilde{Q} \tilde{K}_{m-1}= \tilde{Q} \tilde{H}_{m-1}, \quad \tilde{Q}^H \tilde{Q} = I, \quad \tilde{Q}e_1 = \tilde{w}/\Vert\tilde{w} \Vert_2,
		\end{align*}
		and $\tilde{h}_{i+1,i}/\tilde{k}_{i+1,i} = \tilde{\xi}_i$.
	\end{problem}
	The resulting Hessenberg pencil $(\tilde{H}_{m-1},\tilde{K}_{m-1})$ forms the recurrence relation for ORFs $\{\tilde{r}_\ell\}_{\ell=0}^{m-2}$ satisfying $\tilde{r}_\ell\in\mathcal{R}^{\tilde{\Xi}}_\ell\backslash\mathcal{R}^{\tilde{\Xi}}_{\ell-1}$ and $	\langle \tilde{r}_i,\tilde{r}_j\rangle_{\sim} = \delta_{i,j}$.
	We comment that the order of the poles on the subdiagonal is not
	fixed. One can reorder these poles without any loss of generality, it
	will have no effect on the entire space spanned by the orthogonal
	rational functions represented by the recurrences, but the individual
	orthogonal rational functions, as well as subspaces might of course differ.
	Again we have to note that the pencil we are considering is particular, we
	call it a normal Hessenberg pencil. It admits the special unitary
	factorization, as mentioned in Problem~\ref{problem:HPIEP}. See also
	\cite{rv004} for more properties on normal pencils.

	\section{Manipulating the Hessenberg eigenvalue decomposition}\label{Hess:eig}
	Solving these matrix problems requires us to manipulate the eigenvalue
	decomposition of the associated Hessenberg matrix or pencil. This
	section introduces basic concepts and algorithms that will be used for
	the Hessenberg setting. We work with proper matrices. Being proper implies that the recurrences will not break down.
	In other words: We can construct a full set of orthogonal polynomials
	given the specified inner product. Also the Hessenberg matrix we work
	with is normal by construction, it is thus unitarily diagonalizable.

	In the forthcoming sections we illustrate how to stably modify the
	Hessenberg matrix to remove  the undesired nodes. This
	results in a new proper matrix providing the recurrence
	relations for the modified inner product.
	
	The basic principle is the following one.
	We follow the notation from Problem~\ref{problem:downdateHess}. Suppose we have a solution to Problem~\ref{problem:HIEP}, i.e., $ZQ=QH_m$, where $Q$ is unitary and $Qe_1=w/\|w\|_2$. We construct a unitary matrix\footnote{The
		Hermitian conjugate is put on purpose on the other side, to link to
		the RQ algorithm we will describe further on.} $\hat{Q}$ and
	execute a transformation  on $H_m$: $\hat{Q} Q^H Z Q \hat{Q}^H =
	\hat{Q} H_m \hat{Q}^H = \tilde{H}_m$.
	After some permutations of the rows in $\hat{Q} Q^H$ we can write this equation as 
	\begin{equation}
		\label{eq:beforedefl}
		\begin{bmatrix}
			1 \\
			&\tilde{Q}^H
		\end{bmatrix}
		\begin{bmatrix}
			\tilde{z} \\
			&\tilde{Z}
		\end{bmatrix}
		\begin{bmatrix}
			1 \\
			&\tilde{Q}
		\end{bmatrix}
		=
		\begin{bmatrix}
			\tilde{z} \\
			&\tilde{H}_{m-1}
		\end{bmatrix}.
	\end{equation}
	The permutation is not executed in practical implementations, we have
	simply added it to illustrate that the structure of the permuted
	$\hat{Q} Q^H$ allows to deflate $\tilde{Q}$ easily; $\tilde{H}_m$ is automatically of the correct form.
	Extracting the trailing $(m-1)\times(m-1)$ principal matrices provides us the
	desired solution, under the condition that $\tilde{Q}$ satisfies the
	desired condition on the weight vector, i.e., the first vector of
	$\tilde{Q}$ must be a multiple of $\tilde{w}$. Executing this numerically poses three challenges:
	\begin{enumerate}
		\item Equation~\eqref{eq:beforedefl} must have the correct block form,
		i.e., block diagonal, where $\tilde{z}$ can be identified and
		removed.
		\item The matrix $\tilde{H}_{m-1}$ must be of Hessenberg form.
		\item The desired weight vector must equal the first column of $\tilde{Q}$.
	\end{enumerate}
	


	\subsection{The QR and RQ algorithm for a Hessenberg matrix}
	
	To deflate an eigenvalue we will run a variant of the QR algorithm on a proper Hessenberg matrix.
	Note that we assume the Hessenberg matrix normal and only having
	simple eigenvalues. The version we describe in this paragraph is the
	explicit QR version. Further on, based on Theorem~\ref{theorem:RQ_equivalent} we will	discuss two more numerically different variants.
	
	\begin{definition}[Proper Hessenberg matrix]\label{def:properHess}
		A Hessenberg matrix $H\in\mathbb{C}^{m\times m}$ is called \emph{proper} if all its subdiagonal elements differ from zero, i.e., $h_{i+1,i} \neq 0$ for $i=1,2,\dots,m-1$.
	\end{definition}
	
	The QR algorithm is the method of choice for computing the eigenvalues of modestly sized dense matrices. It is an iterative process, where each step is governed by a shift. The eigenvalue closest to the shift is then pushed to the lower right corner of the matrix and can after sufficient iterations be removed (deflated). This process is repeated until all eigenvalues are found.
	
	Instead of shifted QR steps, we use perfectly shifted RQ steps. QR converges to the lower right corner, RQ converges to the upper left corner; and perfectly shifted means that the shift coincides with an eigenvalue. As a consequence we can expect convergence in a single step to the upper left corner.
	
	An RQ step with shift $\lambda$ works as follows on a Hessenberg
	matrix $H$. Compute the RQ factorization of $H-\lambda I=\hat{R}\hat{Q}$ and form
	a new Hessenberg matrix $\tilde{H}=\hat{Q}\hat{R}+\lambda I$. Essentially this is a
	similarity transformation with $\hat{Q}$ on $H$: $\tilde{H}=\hat{Q}H \hat{Q}^H$.
	For a Hessenberg matrix, a single step of a RQ step with perfect shift
	will reveal the eigenvector corresponding to the shift (eigenvalue)
	$\lambda$, and this is exactly what we will need for the downdating
	algorithms. The theorem and proof considered here stem from
	\cite{p512}; we repeat some parts to make use of these arguments further on. 
	
	\begin{lemma}\label{lemma:RQ_perfectShift}
		Let $H\in \mathbb{C}^{m\times m}$ be a proper Hessenberg matrix with distinct eigenvalues and denote one such eigenvalue by $\lambda\in \mathbb{C}$.
		Then the first row of $\hat{Q}$ appearing in the RQ decomposition of $H-\lambda I$ reveals the eigenvector corresponding to $\lambda$.
	\end{lemma}
	\begin{proof}
		The shifted Hessenberg matrix $H-\lambda I$ is singular.
		Hence, the upper triangular matrix $\hat{R}\in \mathbb{C}^{m\times m}$ from the RQ decomposition
		\begin{equation*}
			H-\lambda I = \hat{R}\hat{Q}
		\end{equation*}
		must also be singular and therefore a diagonal element $r_{i,i}$ must satisfy $r_{i,i}=0$ for some $i$.
		Since $H-\lambda I$ is Hessenberg, its last $m-1$ rows are linearly independent. As also the rows of the unitary matrix $\hat{Q}$ are linearly independent,  it follows that the corresponding rows of $\hat{R}$ must also linearly independent. As a consequence there is only the sole option that $r_{1,1}=0$.
	\end{proof}
	
	From Lemma \ref{lemma:RQ_perfectShift} it follows that a single RQ step with perfect shift leads to the decomposition
	\begin{equation}\label{eq:RQ_perfectShift}
		H-\lambda I = \begin{bmatrix}
			0 & \boldsymbol{\times}\\
			& \hat{R}_{m-1}
		\end{bmatrix} \begin{bmatrix}
			\hat{q}^H\\
			\hat{Q}_{m-1}
		\end{bmatrix}
	\end{equation}
	with $\hat{Q}_{m-1}^H \hat{Q}_{m-1} = I$, $\hat{Q}_{m-1} \hat{q} =
	0$ and $H \hat{q} = \lambda \hat{q}$. The symbol $\boldsymbol{\times}$
	denotes some arbitrary unessential elements, i.e., a row vector in this case.
	The last equation shows that $\hat{q}$ is the right eigenvector of $H$ corresponding to eigenvalue $\lambda$.
	
	\begin{theorem}[Isolate eigenvalue using perfect shift RQ]\label{theorem:downdateHess}
		Let $\hat{R},\hat{Q}$ be the factors obtained by applying an RQ step to a normal $H$, having only simple eigenvalues, shifted with one of its eigenvalues $\lambda$. 
		Then the unitary similarity transformation $\hat{Q} H \hat{Q}^H$ isolates the eigenvalue $\lambda$, allowing deflation.
	\end{theorem}
	\begin{proof}
		This proof appeared in the paper by \cite{p512}.
		Let $Q,Z$ be the factors of the eigenvalue decomposition of the normal matrix $H$, i.e., $H = Q^H Z Q$, $Q^H Q = I$ and $Z = \textrm{diag}(\{z_i\})$.
		From Lemma \ref{lemma:RQ_perfectShift}, more precisely \eqref{eq:RQ_perfectShift}, we have that $\hat{q} = q_j$ for some $j$.
		Now, by relying on the orthogonality of the eigenvectors we obtain
		\begin{align*}
			\hat{q}^H Q^H &=
			\hat{q}^H
			\begin{bmatrix}
				\vert & & \vert & \vert & \vert & & \vert \\
				q_1 & \dots & q_{j-1} & q_j & q_{j+1} & \dots & q_{m}\\
				\vert & & \vert & \vert & \vert & & \vert
			\end{bmatrix} \\
			&= \begin{bmatrix}
				0 & \dots & 0 & \alpha &0 & \dots & 0\\
			\end{bmatrix}\\
			&= \alpha e_j^\top,
		\end{align*}
		where $\alpha = \hat{q}^H q_j \neq 0$.
		Using the notation of above:
		$\hat{Q}:=\begin{bmatrix}
			\hat{q}^H\\
			\hat{Q}_{m-1}
		\end{bmatrix}$, we get
		\begin{align*}
			\hat{Q}_{m-1} Q^H &= \hat{Q}_{m-1}\begin{bmatrix}
				\vert & & \vert & \vert & \vert & & \vert \\
				q_1 & \dots & q_{j-1} & q_j & q_{j+1} & \dots & q_{m}\\
				\vert & & \vert & \vert & \vert & & \vert
			\end{bmatrix}\\
			&= 
			\begin{bmatrix}
				\vert & & \vert &  & \vert & & \vert \\
				\boldsymbol{\times} & \dots & \boldsymbol{\times} & {0} & \boldsymbol{\times} & \dots & \boldsymbol{\times}\\
				\vert & & \vert & & \vert & & \vert
			\end{bmatrix}
		\end{align*}
		This isolates the eigenvalue $\lambda$ corresponding to the eigenvector $\hat{q}$:
		\begin{equation}
			\label{eq:reorder}
			\hat{Q} H \hat{Q}^H = \hat{Q} Q^H Z Q \hat{Q}^H 
			= \begin{bmatrix}
				\lambda & \boldsymbol{0}\\
				\boldsymbol{0} & \widetilde{H}_{m-1}
			\end{bmatrix},
		\end{equation}
		where $\widetilde{H}_{m-1}\in \mathbb{C}^{(m-1)\times (m-1)}$ is an unreduced Hessenberg matrix with spectrum $\sigma(H)\backslash \{\lambda\}$.
		The matrix $\tilde{H}_{m-1}$ must be unreduced, because of the structure of the matrices $\hat{R}$ and $\hat{Q}$, namely $\hat{R}$ must have the trailing $m-1$ diagonal elements different from zero and $\hat{Q}$ is also of proper Hessenberg form.
	\end{proof}
	
	When inserting the permutation $P$ which will switch rows $j$
	and $1$ into Equation~\ref{eq:reorder} we get:
	\begin{equation*}
		(\hat{Q} Q^H P) (P Z P) (P Q \hat{Q}^H) 
		= 
		\begin{bmatrix}
			1 \\
			&\tilde{Q}^H
		\end{bmatrix}
		\begin{bmatrix}
			\tilde{z} \\
			&\tilde{Z}
		\end{bmatrix}
		\begin{bmatrix}
			1 \\
			&\tilde{Q}
		\end{bmatrix}
		= \begin{bmatrix}
			\lambda & \boldsymbol{0}\\
			\boldsymbol{0} & \widetilde{H}_{m-1}
		\end{bmatrix},
	\end{equation*}
	and we can deflate the desired submatrices providing us the solutions
	to Problem~\ref{problem:downdateHess}.
	
	We note that the latter theorem is stronger than the straightforward
	result in Equation~\ref{eq:RQ_perfectShift}, as the upper row is also
	zero. Also important to note is the structure of the matrix $\hat{Q}$, which
	we will use further on as well. When $H$ is a proper Hessenberg
	matrix, the $\hat{Q}$ factor in the QR factorization will also be proper Hessenberg.
	In case we compute the RQ factorization, it is not hard to verify that
	the unitary factor $\hat{Q}$
	will be upper hessenberg as well.


	
	\subsection{Why the RQ algorithm instead of the QR algorithm}
	When applying a perfectly shifted QR step to a proper Hessenberg
	matrix $H_m$ 
	the perfect shift $\tilde{z}$ will appear on position $(m,m)$. Suppose
	the Hessenberg matrix $H_m$ is a solution to the
	Problem~\ref{problem:HIEP}, i.e. $H_m = Q^H Z Q$ and as a result of this QR step we want to
	end up with a solution to Problem~\ref{problem:downdateHess}. Suppose
	the similarity transformation of the QR algorithm is determined by
	$\hat{Q}$. We end up with something like
	\begin{equation*}
		\hat{Q}^H H_m \hat{Q} = \begin{bmatrix}
			\widetilde{H}_{m-1} & \boldsymbol{0}\\
			\boldsymbol{0} & \tilde{z}
		\end{bmatrix}.
	\end{equation*}
	After this transformation the new eigenvector matrix equals
	$Q\hat{Q}$ and can be deflated to obtain the solution to the downdating
	problem. Taking into consideration that the matrix $\hat{Q}$ is upper
	Hessenberg, we see, however, that the first column does not satisfy the weight
	condition anymore:
	\begin{equation*}
		\hat{Q}^H Q^H w = \hat{Q}^H \Vert w \Vert_2 e_1 = \begin{bmatrix}
			\times & \times \\
			\times & \times & \times \\
			\times & \times & \times  & \times \\
			\vdots & \vdots & \vdots & \vdots & \ddots \\
			\times & \times & \times  & \times & \times & \times \\
			\times & \times & \times  & \times & \times & \times 
		\end{bmatrix} \begin{bmatrix}
			\times \\
			0\\
			0\\
			\vdots\\
			0\\
			0
		\end{bmatrix} =  \begin{bmatrix}
			\times \\
			\times\\
			\times\\
			\vdots\\
			\times\\
			\times
		\end{bmatrix}.
	\end{equation*}
	So, the weight condition is no longer satisfied and the obtained
	Hessenberg matrix (after deflation) $\widetilde{H}_{m-1}$ is not a solution to
	the downdated IEP, as noted by \cite{m275}.
	If, however the RQ decomposition is used instead, then the unitary
	factor $\hat{Q}$ in $H_m-\tilde{z}I = \hat{R}\hat{Q}$ is again upper Hessenberg.
	But now the unitary similarity transformation isolating the eigenvalue
	$\tilde{z}$ from $H_m$ differs slightly in the sense that the Hermitian
	conjugates are positioned elsewhere. We get:
	\begin{equation*}
		\hat{Q} H \hat{Q}^H = \begin{bmatrix}
			\tilde{z} & \boldsymbol{0}\\
			\boldsymbol{0} & \widetilde{H}_{m-1}
		\end{bmatrix},
	\end{equation*}
	where $\widetilde{H}_{m-1}$ is again a Hessenberg matrix with the spectrum
	$\sigma(\widetilde{H}_{m-1}) = \sigma(H_m)\backslash\{\tilde{z} \}$ and thus a potential solution to the downdated IEP.
	Let us look at the weight condition is this case. Note that now we multiply the eigenbasis $Q$ from the right by a lower Hessenberg matrix $\hat{Q}^H$, therefore
	\begin{equation}
		\label{eq:weight:constraint}
		\hat{Q} Q^H w = \hat{Q} \Vert w \Vert_2 e_1 = \begin{bmatrix}
			\times & \times & \times  & \times & \times & \times \\
			\times & \times & \times  & \times & \times & \times \\
			& \times & \times  & \times & \times & \times \\
			& 	    & \ddots  & \vdots & \vdots & \vdots \\
			& 	    & 		  & \times & \times & \times \\
			& 	    & 		  & 	   & \times & \times 
		\end{bmatrix} \begin{bmatrix}
			\times \\
			0\\
			0\\
			\vdots\\
			0\\
			0
		\end{bmatrix} =  \begin{bmatrix}
			\times \\
			\times\\
			0\\
			\vdots\\
			0\\
			0
		\end{bmatrix}.
	\end{equation}
	Deflating the matrix $\hat{Q} H_m \hat{Q}^H$ to obtain $\widetilde{H}_{m-1}$,
	results in omitting the first element of $\hat{Q} Q^H w$, which results in $\Vert \tilde{w}\Vert_2 e_1$, with $e_1\in\mathbb{C}^{m-1}$.
	This verifies that $\widetilde{H}_{m-1}$ is a solution to the downdating problem.
	
	\subsection{Theoretical equivalent statements}
	The perfect shift RQ step described above can be implemented in three
	different ways, see \cite{MaVD18}.
	These variants are based on Theorem
	\ref{theorem:RQ_equivalent}. To state the theorem, we first need to
	define a core transformation.
	
	\begin{definition}[Core transformations]\label{def:coreTransformations}
		A core transformation $C_i\in \mathbb{C}^{m\times m}$ is a unitary matrix of the form
		\begin{equation}
			C_i = \begin{bmatrix}
				I_{i-1} \\
				& \times & \times \\
				& \times & \times \\
				& & & I_{m-i-1}\\
			\end{bmatrix},
		\end{equation}
		where $I_{k}$ denotes the identity matrix of size $k\times k$.
	\end{definition}
	Core transformations are essentially $2\times 2$ matrices, since their
	only active part is a $2\times 2$ diagonal block. The parameter $i$ in
	$C_i$ indicates where, on the diagonal, the active block appears, and
	$\mathfrak{C}_i$ denotes the class of all these core transformations.
	
	
	\begin{theorem}[Theorem 2.1 in \cite{MaVD18}]\label{theorem:RQ_equivalent}
		Let $H\in\mathbb{C}^{m\times m}$ be a proper Hessenberg
		matrix and $\lambda \in \sigma(H)$ one of the distinct eigenvalues.
		Then the following statements hold:
		\begin{enumerate}
			\item $H$ has a normalized eigenvector $x$ corresponding to $\lambda$
			\begin{equation*}
				H x = \lambda x,\quad \Vert x \Vert_2 = 1
			\end{equation*}
			which is unique up to unimodular scaling and $e_m^\top x \neq 0$.
			\item An essentially unique sequence of $\{C_i\}_{i=1}^{m-1}$, $C_i \in\mathfrak{C}_i$ forming the matrix
			\begin{equation*}
				\hat{Q} := C_1 C_2 \dots C_{m-1}
			\end{equation*}
			exists that transforms the pair $(H,x)$ to a similar one
			\begin{equation*}
				(\hat{H},\hat{x}) := (\hat{Q} H \hat{Q}^H,\hat{Q} x)
			\end{equation*}
			with
			\begin{equation*}
				\hat{x} = \alpha e_1, \quad \vert \alpha \vert =1, \quad \hat{H} e_1 = \lambda e_1 \text{ and } \hat{H} \text{ is a Hessenberg matrix}.
			\end{equation*}
			\item The Hessenberg matrix $H-\lambda I$ has the RQ decomposition
			\begin{equation*}
				H- \lambda I = \hat{R} \hat{Q},
			\end{equation*}
			where $e_1^\top \hat{R} e_1 = 0$ and $\hat{Q}$ is essentially the same matrix as the one transforming $x$ to $\hat{x}$.
		\end{enumerate}
	\end{theorem}
	
	The proof is due to \cite{MaVD18}, and builds entirely on matrix structures. An
	alternative way to obtain the desired results is looking at
	Krylov subspaces and the associated implicit Q theorem, as
	also stated in their remark. For downdating orthogonal
	rational functions, a similar theorem is required, but the
	proof will build upon rational Krylov subspaces. 
	
	
	
	
	

	\section{Numerical algorithms for downdating polynomials}
	\label{sec:down:poly}
	
	Following Theorem~\ref{theorem:RQ_equivalent} there are three options to compute and execute the similarity transformation. 
	Mathematically they are equivalent, but numerically they behave differently.
	\begin{itemize}
		\item Compute $\hat{Q}$ via the RQ factorization and execute the similarity
		transformation with $\hat{Q}$. This is known as the explicit algorithm.
		\item Compute $\hat{Q}$ via the factorization of the eigenvector and
		execute the similarity transformation with $\hat{Q}$. This is named the eigenvector method in this paper.
		\item Compute the first core transformation $C_1$, from the RQ factorization,
		compute the remaining core transformations in order to restore the
		Hessenberg structure. This is named the implicit algorithm.
	\end{itemize}
	Explicit and implicit QR algorithms are described in various
	textbooks, a thorough reference is \cite{b333}.
	The \emph{eigenvector method}, proposed by \cite{MaVD18},
	is based on statement 2 in Theorem~\ref{theorem:RQ_equivalent} in order to accurately deflate a
	particular eigenvalue. Their procedure leads to a more accurate
	isolation of the given eigenvalue, on the condition that the
	eigenvector is computed with sufficient accuracy.
	
	We compare the three methods, named respectively \emph{explicit matrix method},
	\emph{eigenvector method}, and the \emph{implicit matrix method}.
	In the beginning of Section~\ref{Hess:eig} we mentioned the various
	numerical challenges.
	The three algorithms behave differently, and
	we investigate numerically which method is the most appropriate in the downdating setting.

	\subsection{The explicit matrix method}\label{sec:downdatePoly_matrix}
	The explicit version was used by \cite{m275} to
	downdate a unitary matrix and is
	rather straightforward.
	Compute the RQ decomposition: $H_m-\tilde{z} I = \hat{R}\hat{Q}$, which in exact arithmetic leads to \eqref{eq:RQ_perfectShift} and Theorem \ref{theorem:downdateHess} is valid. 
	Compute the downdated matrix $\widetilde{H}_{m-1}$, which is the
	$(m-1)\times (m-1)$ trailing principal submatrix of $\hat{Q} H \hat{Q}^H = \hat{Q}\hat{R} +
	\tilde{z} I$. 
	More precisely we get, for $\tilde{\epsilon}, \epsilon_z$ and $\epsilon_i$, elements of the size of machine precision,
	\begin{equation}
		\label{eq:H:1}
		\hat{Q}\hat{R} + \tilde{z}I = \begin{bmatrix}
			\tilde{z} + \epsilon_z & {\epsilon}_1 & {\epsilon}_2 & \dots& {\epsilon}_{m-2} & {\epsilon}_{m-1}\\
			\tilde{\epsilon}	&\times & \times & \dots  & \times & \times \\
			&\times & \times & \dots  & \times & \times \\
			&& \times &  \dots  & \times & \times \\
			&& 	 	   & \ddots 		  & \vdots & \vdots \\
			&& 	    & 		& \times & \times \\
		\end{bmatrix}.
	\end{equation}
	The formula $\hat{Q}\hat{R}$ is preferred over $\hat{Q} H_m \hat{Q}^H$ for
	numerical computation, since $\hat{Q}\hat{R}$ will have exact
	Hessenberg structure, whereas $\hat{Q} H_m \hat{Q}^H$ will have small
	elements, due to round-off, on the second subdiagonal. As a
	consequence, only one numerical issue remains, that is the lack of
	block diagonal structure of the transformed matrix:
	We have explicitly put an
	$\tilde{\epsilon}$ to emphasize the numerical error. If
	$\tilde{\epsilon}$ is too large, we run into problems, and we can not
	split the problem into two subproblems. 
	We can remedy this problem, by executing a second step of the explicit
	RQ method, unfortunately we will see that a second step does not only
	double the work, but also perturbs the desired structure of the weight
	vector and even more work is required to restore that structure.
	The resulting Hessenberg matrix is of the form given in
	Equation~\ref{eq:H:1} and executing the transformation $\hat{Q}$, with
	elements $\hat{q}_{ij}$ has
	the following effect on the vector of weights:
	\begin{equation*}
		\hat{Q}\begin{bmatrix}
			\Vert w \Vert_2 \\
			0 \\
			0\\
			0\\
			\vdots\\
			0
		\end{bmatrix} =  \begin{bmatrix}
			\Vert w \Vert_2\, \hat{q}_{1,1} \\
			\Vert w \Vert_2\, \hat{q}_{2,1} \\
			0\\
			0\\
			\vdots\\
			0
		\end{bmatrix}.
	\end{equation*}
	As mentioned before in Equation~\ref{eq:weight:constraint}, this would
	lead to the desired solution of Problem \ref{problem:downdateHess}, since the
	first element will be omitted when deflating the matrix in the bottom
	right corner.
	
	Assume, however, that the element $\tilde{\epsilon}$ in Equation~\ref{eq:H:1} is not
	small enough, and there is the need to execute another step of the
	explicit QR algorithm to reduce the size of this element\footnote{Note that this is typically the case in practical QR algorithms. It requires more than one step to get convergence.}.
	Let us for simplicity denote the extra similarity to be executed with
	the matrix $\hat{\hat{Q}}$. We get, denoting the elements of
	$\hat{\hat{Q}}$ by $\hat{\hat{q}}_{ij}$:
	\begin{equation}
		\label{eq:ww}
		\hat{\hat{Q}}  \hat{Q}\begin{bmatrix}
			\Vert w \Vert_2 \\
			0 \\
			0\\
			0\\
			\vdots\\
			0
		\end{bmatrix} =  \begin{bmatrix}
			\Vert w \Vert_2\, \hat{q}_{1,1}\hat{\hat{q}}_{11}  +
			\Vert w \Vert_2\, \hat{q}_{2,1} \hat{\hat{q}}_{12} \\
			\Vert w \Vert_2\, \hat{q}_{1,1}\hat{\hat{q}}_{21}  +
			\Vert w \Vert_2\, \hat{q}_{2,1} \hat{\hat{q}}_{22} \\
			\Vert w \Vert_2\, \hat{q}_{2,1} \hat{\hat{q}}_{23} \\
			0\\
			\vdots\\
			0
		\end{bmatrix}.
	\end{equation}
	
	If the resulting Hessenberg matrix $\hat{\hat{H}}$ has a sufficiently small element in position $(2,1)$ we can deflate the Hessenberg matrix in
	the lower right corner. Unfortunately after deflating \eqref{eq:ww} the resulting vector will have two nonzero elements instead of a single one in the vector on the right-hand side.
	
	Let us examine the structure of the matrices in more detail. After
	having executed both transformations and having executed the
	deflation we end up with $\tilde{Q}^H \tilde{Z} \tilde{Q} =
	\tilde{H}_{m-1}$, but unfortunately, we see that $\tilde{Q}^H \tilde{w} \neq
	\|\tilde{w}\|_2 e_1$.	
	As a consequence we need a procedure, or another unitary
	similarity transformation, say $C$, such that 
	$\tilde{\tilde{Q}}^H \tilde{Z} \tilde{\tilde{Q}} = C^H \tilde{Q}^H \tilde{Z} \tilde{Q} C =
	\tilde{\tilde{H}}_{m-1}$, such that $C^H \tilde{Q}^H \tilde{w} =
	\|\tilde{w}\|_2 e_1$.
	
	This procedure is very similar to the Hessenberg updating procedure already
	described by \cite{VBVBVa22}. Let us sketch the
	algorithm on a high level.  The vector $\tilde{Q}^H \tilde{w}$ has two leading
	nonzero elements. A single rotation $C_1^H$ can transform
	$\tilde{Q}^H\tilde{w}$ to a multiple of $e_1$. We get $C_1^H
	\tilde{Q}^H\tilde{w}=\|\tilde{w}\|_2 e_1$. Executing the similarity transformation $C_1^H \tilde{Q}^H \tilde{Z} \tilde{Q} C_1$
	disturbs the Hessenberg structure: the element $\otimes$ gets
	introduced 
	\begin{equation*}
		K_1=
		C_1^H \tilde{Q}^H \tilde{Z} \tilde{Q} C_1
		=
		\left[\begin{array}{cccccc}
			\times & \times & \times & \times & \cdots & \times \\
			\times & \times & \times & \times & \cdots & \times \\
			\otimes & \times & \times & \times & \cdots &\times \\
			& & \times & \times & \cdots &\times \\
			& & & \ddots & & \vdots
		\end{array}\right].
	\end{equation*}
	The unitary transformation $C_2$ is constructed to operate on rows $2$
	and $3$ and annihilate the $\otimes$ element by operating on the
	left, i.e. $C_2^H C_1^H \tilde{Q}^H \tilde{Z} \tilde{Q} C_1$ is again
	of Hessenberg form.  Also $C_2^H C_1^H
	\tilde{Q}^H\tilde{w}=\|\tilde{w}\|_2 e_1$ still holds, as $C_2$ operates
	on rows $2$ and $3$. But, executing a similarity imposes a
	multiplication with $C_2$ on the right of
	the matrix $K_1$ as well. As a structure we get
	\begin{equation*}
		K_2=
		C_2^H C_1^H \tilde{Q}^H \tilde{Z} \tilde{Q} C_1 C_2
		=
		\left[\begin{array}{cccccc}
			\times & \times & \times & \times & \cdots & \times \\
			\times & \times & \times & \times & \cdots & \times \\
			& \times & \times & \times & \cdots &\times \\
			& \otimes & \times & \times & \cdots &\times \\
			& & & \ddots & & \vdots
		\end{array}\right].
	\end{equation*}
	Clearly this procedure can be continued, until the bulge, marked with
	$\otimes$ slides off the matrix. To fully restore the structure we need
	$m-2$ similarity transformations. Denoting $C=C_1\cdots C_{m-2}$ and $\tilde{\tilde{H}}_{m-1}=K_{m-2}$ we
	see that we have solved Problem~\ref{problem:downdateHess}.

		\subsection{The implicit matrix method}
	To initiate the implicit method, we compute the trailing row of
	$H_m-\lambda I$. Because of the Hessenberg structure, the vector $e_m^T(H_m-\lambda
	I)=v^T$ has only two nonzero elements in positions $m-1$ and $m$. To compute the $RQ$
	factorization of $H_m-\lambda I$ a rotation $C_{m-1}\in\mathfrak{C}_{m-1}$ must be created,
	such that $v^T C_{m-1}= \|v\|_2 e_m^T$.
	
	Next the similarity transformation with $C_{m-1}$ is executed on $H_m$. We end up with
	$C_{m-1} H_m C_{m-1}^H$. The resulting matrix is not exactly of
	Hessenberg form anymore. It is perturbed in the second subdiagonal, in
	the last row.
	\begin{equation*}
		\left[\begin{array}{ccccccc}
			\ddots & &&  & & \vdots \\
			& \times &   \times      & \times & \times & \times \\
			&   &  \times      & \times & \times & \times \\
			&    &       & \times & \times & \times \\
			&     &      &\otimes & \times & \times
		\end{array}\right].
	\end{equation*}
	Now we aim at restoring the structure of the Hessenberg matrix, in a
	similar manner as we did in Section~\ref{sec:downdatePoly_matrix},
	but now with the difference that we chase upwards, instead of
	downwards. Theorem~\ref{theorem:RQ_equivalent} then proves that after restoring
	the Hessenberg structure, we have executed one RQ step. By
	construction the Hessenberg matrix is numerically of Hessenberg
	form. Hence the only numerical issue is identical to the one in the
	explicit QR method, the off-diagonal element might not be small
	enough to get the block diagonal structure. Nevertheless, we can do exactly the same trick as in the
	explicit RQ case, that is, do one more step of the implicit method,
	and then restore the structure of the weight vector.

	\subsection{The eigenvector method}\label{sec:downdatePoly_MaVD}
	\cite{MaVD18} propose an alternative
	procedure to perform a perfectly shifted RQ step  leading to a more
	accurate construction of the rotations being executed in the
	similarity transformation. This leads to a more accurate isolation of
	the given eigenvalue in a single step. Recall that in the explicit and
	implicit algorithm typically two steps need to be executed followed by
	a procedure to restore the link between the unitary transformation and
	the weight vector.\\
	The eigenvector method consists of three steps.
	Assume we have a number $\tilde{z}$, close enough to an eigenvalue of
	the matrix $H_m$ (which in theory should be perfect in our setting).
	The first step is to compute the (unit) eigenvector $x$ such that $(H_m-\tilde{z}
	I)x$ is sufficiently small, that is
	\begin{equation}\label{eq:evec_quality}
		\Vert (H_m-\tilde{z}I)x\Vert_2 \approx \epsilon_{\textrm{mach}} \Vert H_m - \tilde{z}I \Vert_2.
	\end{equation}
	In our setting the eigenvector $x$ can be computed from the matrix of recurrences $H_m$ and the weight vector $w$.
	The relation between eigenvector $x$ of $H_m$ and the sequence of OPs $\{p_j\}_{j=0}^{m-1}$ evaluated in the corresponding eigenvalue $\tilde{z}$ is described by, e.g., \cite{GoWe69}. 
	If $\tilde{z} = z_j$ and $w_j$ the corresponding weight, then $x = \bar{w}_j \begin{bmatrix}
		p_0(z_j) & p_1(z_j) & \dots & p_{m-1}(z_j)
	\end{bmatrix}^H$ and the polynomials $p_j(z)$ can be evaluated in $z=z_j$ by running the recurrence relation given by the Hessenberg matrix $H_m$.
	Afterwards, if condition \eqref{eq:evec_quality}, more precisely we will use $\Vert (H_m-\tilde{z}I)x\Vert_2 \leq 3 \epsilon_{\textrm{mach}} \Vert H_m - \tilde{z}I \Vert_2$, is not satisfied, the accuracy of $x$ can be improved by performing several steps of iterative refinement.
	The second step uses the eigenvector $x$  to compute an eigenvector $\dot{x}$ satisfying
	\begin{equation}\label{eq:cond_47}
		\left\Vert \begin{bmatrix}\hat{\epsilon}_1 & \hat{\epsilon}_2 & \dots &\hat{\epsilon}_m		\end{bmatrix}\right\Vert_2 \leq \epsilon_{\textrm{mach}} \Vert H_{m}\Vert_\textrm{F},
	\end{equation}
	 with  $\hat{\epsilon}_i = \frac{e_i^\top ((H_{m}-\tilde{z}I)\dot{x})}{\left\Vert \begin{bmatrix}
		\dot{x}_{i-1} & \dot{x}_i & \dots & \dot{x}_{m}
	\end{bmatrix} \right\Vert_2}$ and the convention $\dot{x}_{0} := 0$.
	Condition \eqref{eq:cond_47} is a way to measure the accuracy of the eigenvector $\dot{x}$ which places greater importance in the trailing entries.
	This is necessary because the perfect RQ shift requires an eigenvector that is especially accurate in its trailing entries, for details and how to obtain such an eigenvector we refer to \cite{MaVD18}.
	In the third step, which is the perfect RQ step, the unitary similarity transformation is executed.
	Transform the eigenvector $\dot{x}$ to $e_1 = \begin{bmatrix}
		1 & 0 & \dots & 0
	\end{bmatrix}^\top$ by a sequence of rotations, i.e., compute $C_i\in\mathfrak{C}_i$ such that $\prod_{i=m-1}^{1} C_i  x = e_1$. 
	Set $\hat{Q} = \prod_{i=1}^{m-1} C_i$, then by Theorem~\ref{theorem:RQ_equivalent}
	it remains to execute  a unitary similarity transformation with
	$\hat{Q}$ in order to get
	\begin{equation*}
		\hat{Q} H_m \hat{Q}^H =\begin{bmatrix}
			\tilde{z} & \boldsymbol{0}\\
			\boldsymbol{0} & \widetilde{H}_{m-1}
		\end{bmatrix}.
	\end{equation*}

	\section{Numerical experiments}\label{sec:numExp:poly}
	The three numerical algorithms are used to solve Problem \ref{problem:downdateHess} and compared in terms of four metrics.
	Starting from a solution $(H_m,Q)$ to Problem~\ref{problem:HIEP}, $\ell$ nodes are downdated one by one, resulting after each downdating step in the solutions $\tilde{H}_{m-1},\tilde{H}_{m-2},\dots, \tilde{H}_{m-\ell}$ to Problem~\ref{problem:downdateHess}.
	The four metrics measure how accurate $\tilde{H}_{m-k}$ represents a sequence of OPs.
	For the computation of these metrics we also need to compute the unitary matrices in each downdating step. We denote these as $\tilde{Q}_{m-1},\tilde{Q}_{m-2},\dots,\tilde{Q}_{m-\ell}$.
	We stress that the computation of $Q$ and $\tilde{Q}_{m-k}$ is not required in any of the three numerical algorithms, their only purpose is to compute the following metrics.
	Denote by $\tilde{Z}$ the appropriately redefined matrix of nodes, i.e.,  $Z$ where the downdated nodes are omitted.
	\begin{enumerate}
		\item The \textit{orthogonality error} measures the orthogonality of the basis represented by $\tilde{Q}_{m-k}$:
		\begin{equation}\label{eq:err_o}
			\textrm{err}_{\textrm{o}} := \Vert \tilde{Q}_{m-k}^H \tilde{Q}_{m-k} - I\Vert_2.
		\end{equation}
		\item The \textit{recurrence error} is a metric for the accuracy of the recurrence coefficients in the matrix of recurrences $\tilde{H}_{m-k}$:
		\begin{equation*}\label{eq:err_r}
			\textrm{err}_{\textrm{r}} := \frac{\Vert \tilde{Z} \tilde{Q}_{m-k} - \tilde{Q}_{m-k} \tilde{H}_{m-k}\Vert _2}{\max\left(\Vert \tilde{Z}\tilde{Q}_{m-k} \Vert_2, \Vert \tilde{Q}_{m-k} \tilde{H}_{m-k} \Vert_2 \right)}.
		\end{equation*}
		\item The \textit{weight error} compares the given weight vector $\tilde{w}/\Vert\tilde{w}\Vert_2$ with the weights $\tilde{Q}_{m-k} e_1$: 
		\begin{equation}\label{eq:err_w}
			\textrm{err}_\textrm{w} := \Vert \Vert \tilde{w} \Vert_2 \tilde{Q}_{m-k} e_1 - \tilde{w} \Vert_2
		\end{equation}
		\item The \textit{node error} quantifies how close the eigenvalues of the matrix of recurrences $\tilde{H}_{m-k}$ are to the diagonal elements in $\tilde{Z}$. Denote the eigenvalues of $\tilde{H}_{m-k}$ by $\{\lambda_j\}_{j=1}^{m-k}$, sorted such that they correspond to $e_j^\top\tilde{Z}e_j \approx \lambda_j$, then the node error is defined as:
		\begin{equation}\label{eq:err_nodes}
			\textrm{err}_{\textrm{node}} = \max_{j} \vert e_j^\top \tilde{Z} e_j-\lambda_j\vert.
		\end{equation}
	\end{enumerate}
	The orthogonality and recurrence error provide an indication for the quality of the orthogonal polynomials that are represented by $\tilde{Q}_{m-k}$ and $\tilde{H}_{m-k}$, their orthogonality and the quality of the recurrence coefficients, respectively.
	The inner product that is encoded in $\tilde{H}_{m-k}$ and $\tilde{Q}_{m-k}$, i.e., the discrete inner product with as nodes the eigenvalues of $\tilde{H}_{m-k}$ and as weights the first column of $\tilde{Q}_{m-k}$, is compared to the inner product defined by the given nodes and weights by means of the weight and node error.\\ 
	The matrix methods are implemented as they are described in Section \ref{sec:down:poly}. 
	For a matrix method with 2 steps, elements denoted by $\otimes$ which are (numerically) eliminated during the bulge chase are explicitly set to zero.
	This is important, since this guarantees that we end up with a matrix that has Hessenberg structure and therefore represents a sequence of OPs.\\
	The behavior of the eigenvector method is studied in more detail by tracking the two conditions in Equation \ref{eq:evec_quality} and Equation \ref{eq:cond_47}.
	In the remainder of this section we apply the numerical algorithms to three experiments.
	We will consider the following methods: the explicit matrix method with 1 step (${\color{blue}\ast}$), the implicit matrix method with 1 step (${\color{red}\circ}$) and 2 steps (${\color{black}\circ}$) and the eigenvector method with $n_{\textrm{IR}}$ as the maximum number of iterative refinement steps (${\color{green}+}$).

	\subsection{Unit circle}\label{sec:exp:polyUC}
	An experiment performed by \cite{VBVBVa22} in the context of updating inverse eigenvalue problems is repeated here.
	Consider $m$ nodes chosen equidistant on the unit circle $\{z_j\}_{j=1}^m$, they are chosen in a balanced way by always taking the next one as far away from the other ones as possible.
	This is illustrated in Figure \ref{fig:balancedCircle} for $m=4,5$ and $6$.
	\begin{figure}[!ht]
		\centering
		\includegraphics{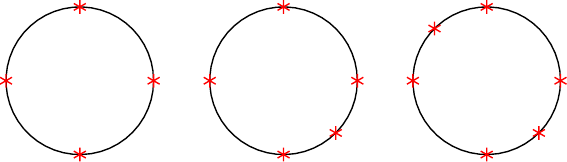}
		\caption{Balanced order of placing $m$ equidistant nodes ${\color{red} \ast}$ on a circle for $m=4,5$ and $6$.}
		\label{fig:balancedCircle}
	\end{figure}
	We start from an available solution $H_m,Q_m\in\mathbb{C}^{m\times m}$ that is obtained by applying the Arnoldi iteration to the matrix $\textrm{diag}(z_1,\dots,z_m)$ and vector $w = \frac{1}{\sqrt{m}}\begin{bmatrix}
		1&
		\dots&
		1
	\end{bmatrix}^\top$.
	We downdate $m/2$ nodes in the reverse balanced order, which leaves $m/2$ equidistant nodes on the unit circle.
	For $m=500$, $n_{\textrm{IR}}=1$, the metrics for $\tilde{H}_{m-k}$ are shown in Figure \ref{fig:UC_opt_m500}.
	\begin{figure}[!ht]
		\centering
		\includegraphics{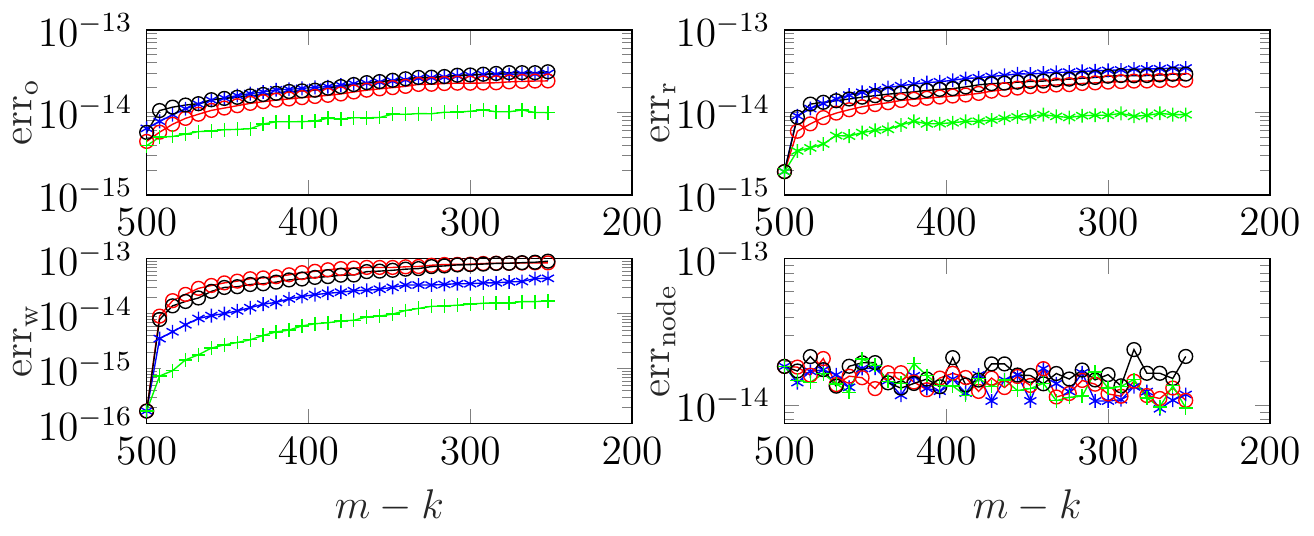}
		\caption{Unit circle experiment with $m=500$, $n_{\textrm{IR}}=1$ and a balanced order of downdating. Metrics comparing the explicit matrix method: ${\color{blue}\ast}$, implicit matrix method: 1 step $\color{red} \circ$, 2 steps $\color{black}\circ$, and eigenvector method with $n_{\textrm{IR}}=1$: $\color{green}+$.}
		\label{fig:UC_opt_m500}
	\end{figure}
	The metric $\textrm{err}_\textrm{node}$ is very similar for all methods and indicates that the nodes are preserved equally well for all methods.
	When comparing the methods for the other three metrics it is clear that the eigenvector method performs best.
	In Figure \ref{fig:UC_opt_m500_emeth} on the left we see that there are given eigenvectors which are far from satisfying \eqref{eq:evec_quality}, however, we succeed for almost all of these given eigenvectors in producing an eigenvector $\dot{x}$ satisfying \eqref{eq:cond_47}, shown on the right of the figure.
	
	\begin{figure}[!ht]
		\centering
		\includegraphics{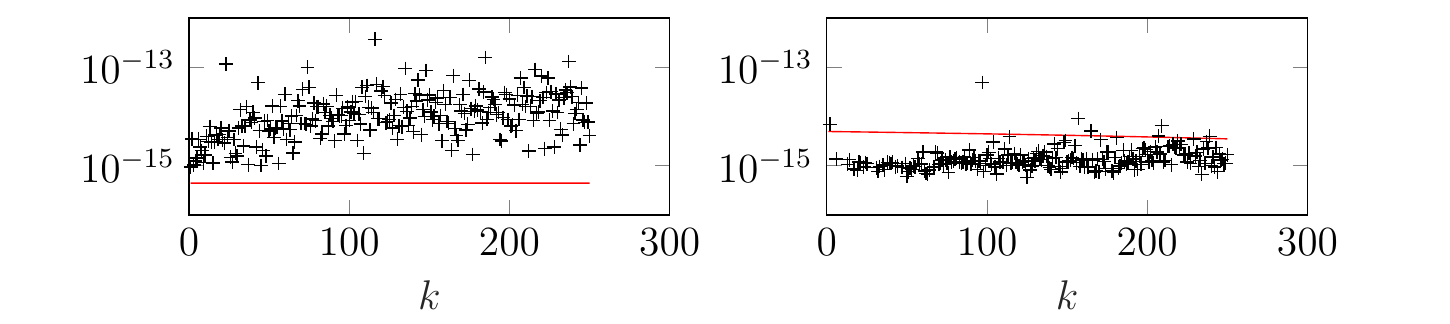}
		\caption{Unit circle experiment with $m=500$, $n_{\textrm{IR}}=1$ and a balanced order of downdating. The conditions \eqref{eq:evec_quality} and  \eqref{eq:cond_47}, with the left hand side ${\color{black} + }$ and right hand side ${\color{red}-}$.}
		\label{fig:UC_opt_m500_emeth}
	\end{figure}
	The implicit method with two steps does not perform better than the implicit method with one step when comparing the metrics shown in Figure \ref{fig:UC_opt_m500}.
	However, it produces a coefficient matrix $\tilde{H}_{m/2}$ with more accurate matrix properties: Namely, for nodes on the unit circle the resulting Hessenberg coefficient matrix is unitary.
	In Table \ref{table:UC} we compare how close the generated matrix $\tilde{H}_{m/2}$ is to a unitary matrix by the quantity $\Vert \tilde{H}_{m/2}^H \tilde{H}_{m/2} - I\Vert_2$.
	There is a clear improvement for the implicit matrix method when executing two RQ steps instead of one.
	\begin{table}[!ht]
		\renewcommand{\arraystretch}{1.3}
		\begin{tabular}{l|cccc}
			& Explicit    & Implicit    & Implicit 2 steps & Eigenvector \\ \hline
			{$\Vert \tilde{H}_{m/2}^H \tilde{H}_{m/2} - I\Vert_2 $} & $2.79\mathrm{e}{-14}$ & $3.44\mathrm{e}{-14}$ & $1.59\mathrm{e}{-14}$        & $1.63\mathrm{e}{-14}$
		\end{tabular}
		\caption{Metric $\Vert \tilde{H}_{m/2}^H \tilde{H}_{m/2} - I\Vert_2 $ measuring the orthogonality of the Hessenberg matrix $\tilde{H}_{m/2}$ obtained by downdating $250$ nodes of a unitary Hessenberg matrix $H_m$ with $m=500$ using four methods.}
		\label{table:UC}
	\end{table}
	
	\subsection{Chebyshev nodes}
	In the following experiment we use a matrix of recurrences $H_m\in \mathbb{R}^{m\times m}$ that is known analytically.
	Namely, for weights $w_j = 1$, $j=1,\dots,m$, and Chebyshev nodes $\{z_j\}_{j=1}^{m}$, i.e., the roots of Chebyshev polynomials of the first kind $z_j = \cos\left(\frac{\pi (j-1/2)}{m}\right)$, the matrix is
	\begin{equation*}
		H_m = \begin{bmatrix}
			0 & 1/\sqrt{2} \\
			1/\sqrt{2} & 0 & 1/2\\
			& 1/2 & 0 & 1/2\\
			& & 1/2 & 0 & \ddots\\
			& & & \ddots & \ddots & 1/2\\
			& & & & 1/2 & 0
		\end{bmatrix}.
	\end{equation*}
	Half of the nodes are downdated from the inner product, the downdated nodes are chosen in two ways, balanced and unbalanced. 
	The balanced  choice is obtained by interpreting $z_j$ as the real part of a point on the upper half of the unit circle and then downdate in the same manner as is described in the above experiment.
	The unbalanced choice is unbalanced in the sense that there will be nodes close to each other on the unit circle which are downdated consecutively.\\
	For $m=500$, $n_{\textrm{IR}}=1$ and a balanced choice Figure~\ref{fig:cheb_opt_m500} compares the four methods, again the eigenvector method performs best.
	\begin{figure}[!ht]
		\centering
		\includegraphics{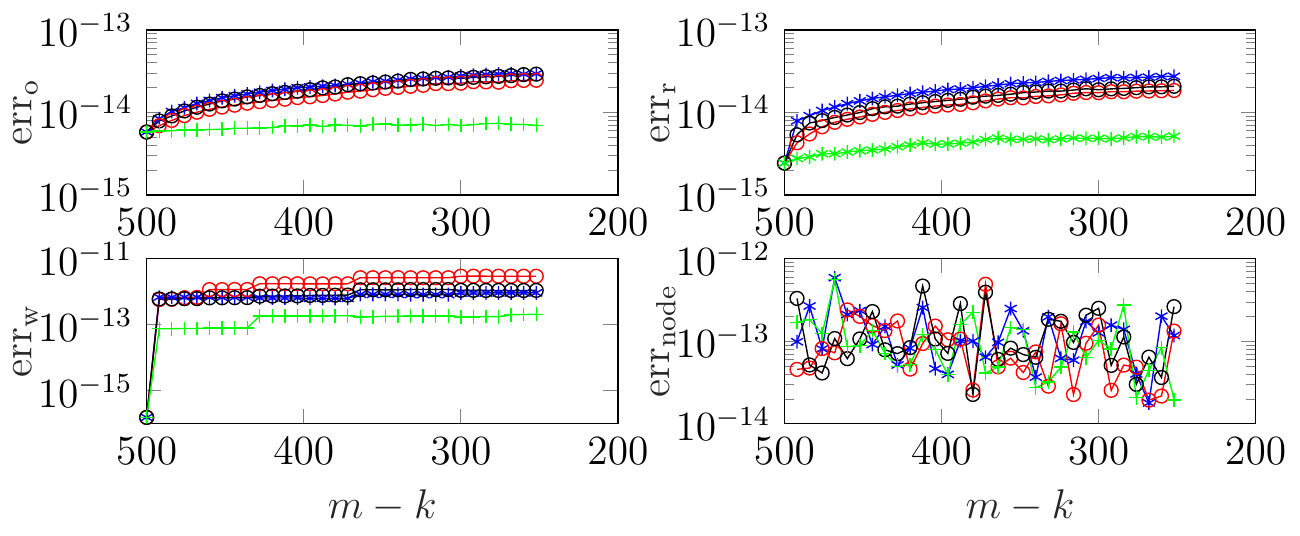}
		\caption{Chebyshev experiment with $m=500$ and a balanced order of downdating. Metrics comparing the explicit matrix method: ${\color{blue}\ast}$, implicit matrix method: 1 step $\color{red} \circ$, 2 steps $\color{black}\circ$, and eigenvector method with $n_{\textrm{IR}}=1$: $\color{green}+$.}
		\label{fig:cheb_opt_m500}
	\end{figure}
	
	Since all nodes are real, the resulting Hessenberg matrix must be a tridiagonal matrix.
	In Table \ref{table:tridiag} the Euclidean norm of the matrix $\tilde{H}_{m/2}$ after omitting its diagonal and first sub-and superdiagonal is used to quantify how close $\tilde{H}_{m-k}$ is to a tridiagonal matrix.
	\begin{table}[!ht]
		\renewcommand{\arraystretch}{1.3}
		\begin{tabular}{l|cccc}
			& Explicit     & Implicit     & Implicit 2 steps & Eigenvector  \\ \hline
			$\Vert$\verb|triu(|$\tilde{H}_{m/2}$\verb|,2)|$\Vert_2$ & $1.46\mathrm{e}{-14}$ & $1.40\mathrm{e}{-15}$ & $3.69\mathrm{e}{-16}$     & $2.13\mathrm{e}{-15}$
		\end{tabular}
		\caption{Chebyshev experiment with $m=500$, and $n_{\textrm{IR}}=1$ after $m/2$ steps of downdating with a balanced choice for nodes. Nearness of $\tilde{H}_{m/2}$ to being tridiagonal is shown, by taking two-norm of the matrix $\tilde{H}_{m/2}$ after omitting its diagonal and first sub-and superdiagonal.}
		\label{table:tridiag}
	\end{table}
	
	For the next experiment we choose $m=200$, $n_\textrm{IR} = 1$ and an unbalanced order.
	Figure~\ref{fig:cheb_subopt_m200} shows a breakdown of the eigenvector method at $k=69$ and a loss of accuracy, for the metrics $\textrm{err}_\textrm{r}$ and $\textrm{err}_\textrm{w}$, at $k=57$.
	\begin{figure}[!ht]
		\centering
		\includegraphics{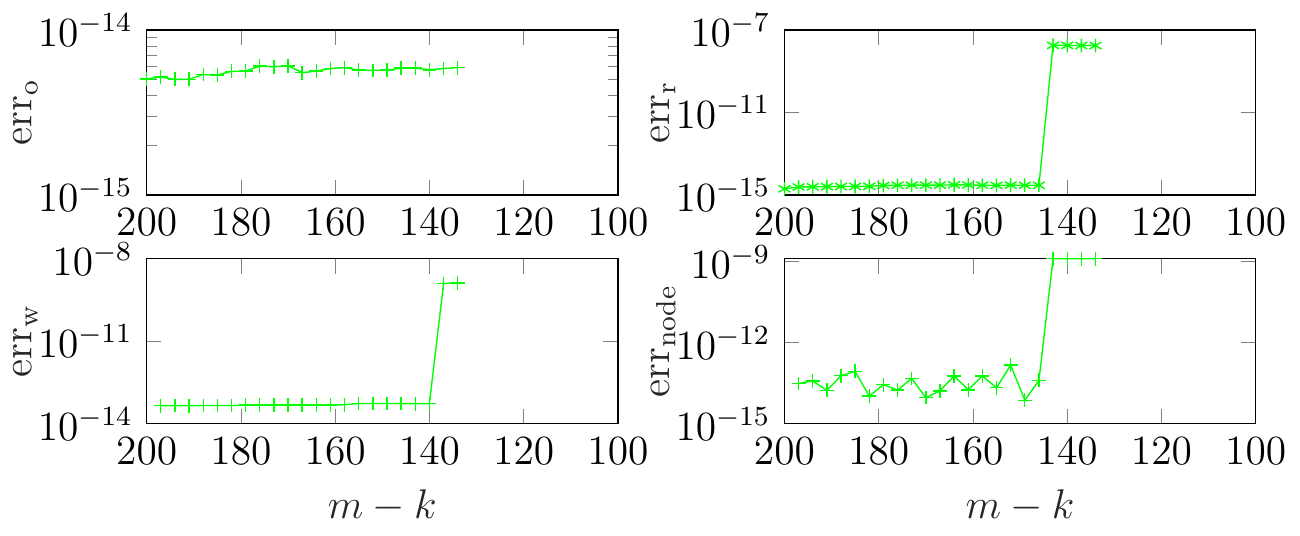}
		\caption{Chebyshev experiment with $m=200$ and a unbalanced order of downdating. Metrics for the eigenvector method with $n_{\textrm{IR}}=1$: $\color{green}+$. A breakdown occurs at $k=69$.}
		\label{fig:cheb_subopt_m200}
	\end{figure}
	
	In Figure \ref{fig:cheb_subopt_m200_emeth} at index $k=57$ we see that the given eigenvector is only accurate up to $10^{-9}$ and leads to an inaccurate $\dot{x}$, which does not satisfy \eqref{eq:cond_47}.
	This causes the eigenvector method to break down a few steps later.
	\begin{figure}[!ht]
		\centering
		\includegraphics{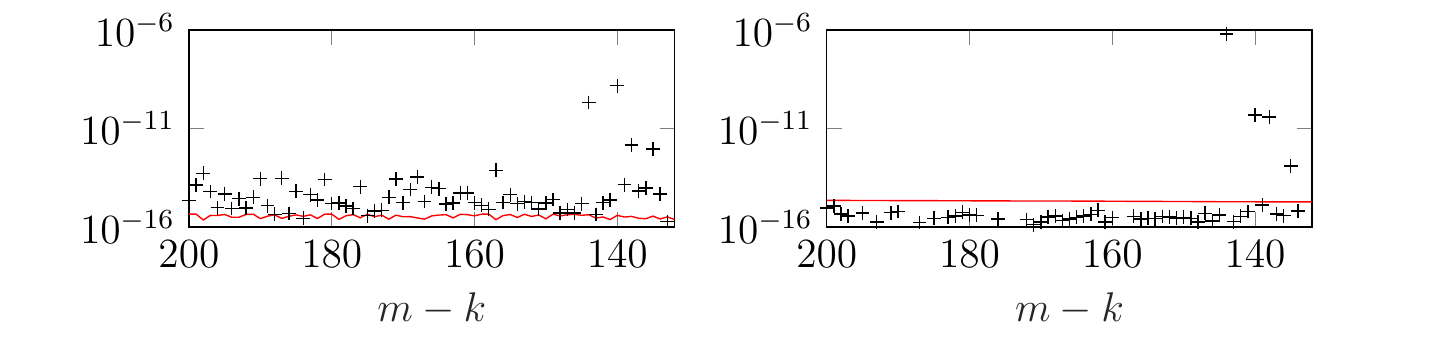}
		\caption{Chebyshev experiment with $m=200$, $n_{\textrm{IR}}=1$ and an unbalanced order of downdating. The conditions \eqref{eq:evec_quality} and  \eqref{eq:cond_47}, with the left hand side ${\color{black} + }$ and right hand side ${\color{red}-}$. }
		\label{fig:cheb_subopt_m200_emeth}
	\end{figure}
	To improve the accuracy of this given eigenvector $x$ we allow $n_\textrm{IR} =2$ steps of iterative refinement, the corresponding quantities of the eigenvector method are shown in Figure \ref{fig:cheb_subopt_m200_emeth_nIR2}.
	The initial eigenvectors are now computed more accurately, with accuracy of $10^{-11}$ or better which appears to be sufficient for the eigenvector method to successfully downdate half of the nodes.
	This means that allowing more iterative refinement steps improves the robustness of the eigenvector method.
	\begin{figure}[!ht]
		\centering
		\includegraphics{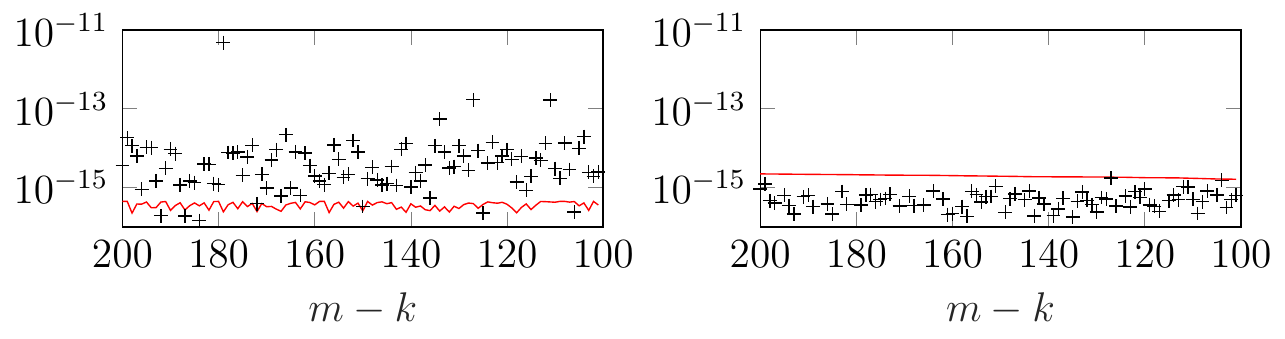}
		\caption{Chebyshev experiment with $m=200$, $n_{\textrm{IR}}=2$ and unbalanced order of downdating. The conditions \eqref{eq:evec_quality} and  \eqref{eq:cond_47}, with the left hand side ${\color{black} + }$ and right hand side ${\color{red}-}$.}
		\label{fig:cheb_subopt_m200_emeth_nIR2}
	\end{figure}
	The metrics shown in Figure~\ref{fig:cheb_subopt_m200_nIR2} reveal that the eigenvector method is again the method of choice.
	\begin{figure}[!ht]
		\centering
		\includegraphics{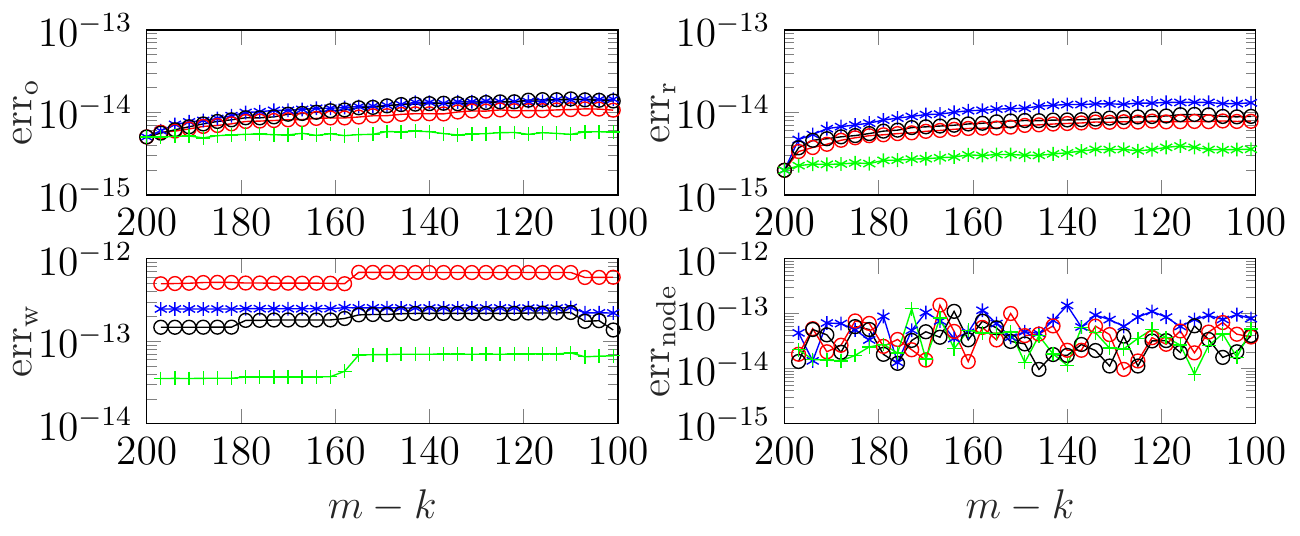}
		\caption{Chebyshev experiment with $m=200$ and an unbalanced order of downdating. Metrics comparing the explicit matrix method: ${\color{blue}\ast}$, implicit matrix method: 1 step $\color{red} \circ$, 2 steps $\color{black}\circ$, and eigenvector method with $n_{\textrm{IR}}=2$: $\color{green}+$.}
		\label{fig:cheb_subopt_m200_nIR2}
	\end{figure}

	\subsection{Equidistant points}
	In many real world applications data is measured in equidistant nodes.
	For a least squares problem formulated using this data, the associated inner product has equidistant nodes.
	We take $m$ equidistant nodes $\{z_j\}_{j=1}^m$, with $z_j = j/m$, and downdate $m/2$ nodes.
	The experiment setup is $m=250$ and $w=\frac{1}{\sqrt{m}}\begin{bmatrix}
		1 & \dots &1
	\end{bmatrix}^\top$ with the following choice of nodes for downdating, all nodes with even index are downdated by alternating between the smallest and the largest remaining index, i.e., in the order: $2,m,4,m-2,6,m-4,\dots$ (these are the indices for the initial set of $m$ nodes).
	To obtain a sufficiently accurate eigenvector for the eigenvector method it is necessary to allow more steps of iterative refinement $n_{\textrm{IR}}=10$ and to execute $b=5$ iterative refinement steps before checking condition \eqref{eq:evec_quality} again.
	For $b=1$ some of the eigenvectors $x$ are not computed accurately enough, leading to a breakdown, exploring this phenomenon and how to choose $b$ appropriately is subject of future research, but it shows that the eigenvector method is sensitive to the given eigenvector.
	Figure \ref{fig:equi_m250_LR} shows the metrics of the four downdating methods, the eigenvector method provides the best results for the metrics used in this paper.
	However, it is necessary to compute the given eigenvector accurately enough.
	\begin{figure}[!ht]
		\centering
		\includegraphics{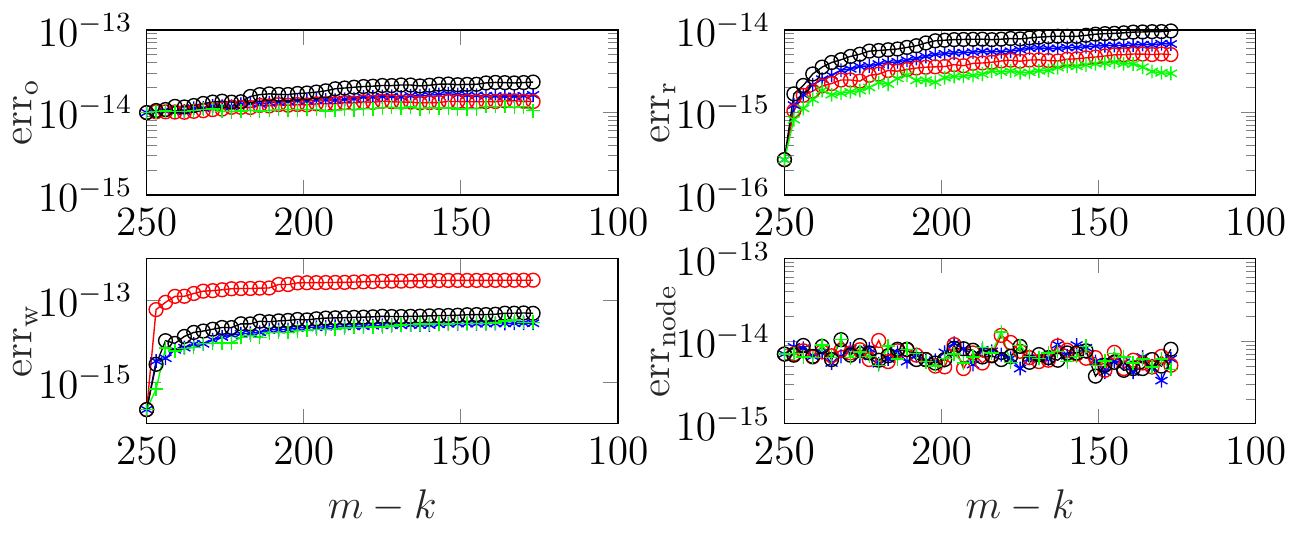}
		\caption{Equidistant nodes experiment with $m=250$. Metrics comparing the explicit matrix method: ${\color{blue}\ast}$, implicit matrix method: 1 step $\color{red} \circ$, 2 steps $\color{black}\circ$, and eigenvector method with, $n_{\textrm{IR}}=10$ and $b=5$: $\color{green}+$.}
		\label{fig:equi_m250_LR}
	\end{figure}
	\begin{figure}[!ht]
		\centering
		\includegraphics{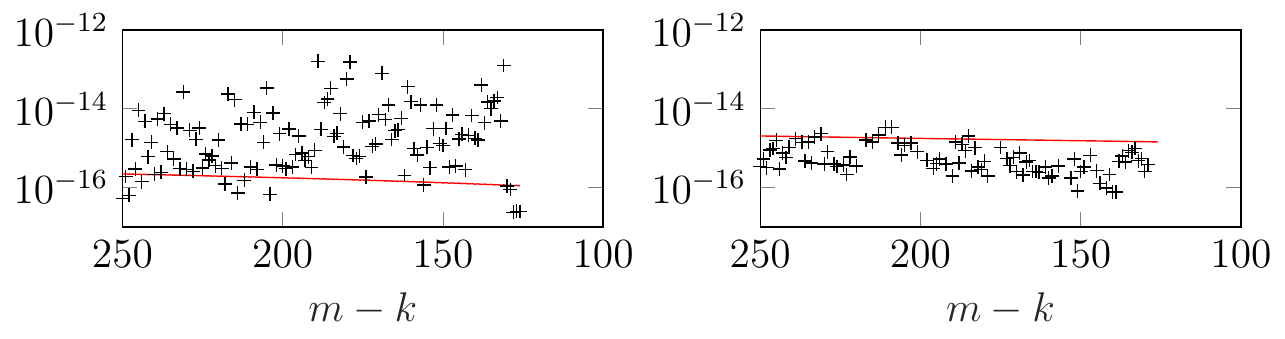}
		\caption{Equidistant nodes experiment with $m=250$, $n_{\textrm{IR}}=10$ and $b=5$. The conditions \eqref{eq:evec_quality} and  \eqref{eq:cond_47}, with the left hand side ${\color{black} + }$ and right hand side ${\color{red}-}$.}
		\label{fig:equi_m250_emeth_LR}
	\end{figure}	
	
	\newpage
	\section{QR and Krylov}\label{sec:QRKrylov} 
	In the polynomial setting we could rely on the explicit RQ (QR) algorithm
	and the structure of the factorizations involved to locate the desired
	eigenvectors and formulate theorems on structure preservation. For the
	rational RQ (QR) setting, things are more involved and the explicit
	version is not readily available. In order to derive the essential
	theorems, we will rely on Krylov subspaces. Typically the link between
	Krylov subspaces is explained for the QR, so in this section, to follow the
	classical conventions we will switch to the QR algorithm instead of
	the RQ one. The RQ is, however, closely linked to QR via an upside
	down and left right flip of the matrix on which we are operating,
	i.e., suppose we want to execute an RQ step on $H$. Consider the upper
	Hessenberg matrix $F=JH^T J$, with $J$ the counter identity. By using
	the QR factorization of $F=QR$  we can get an RQ factorization of
	$H=(JR^TJ)\, (JQ^T J)$. So the theory we will deduce here for the QR
	links directly to the RQ.
	
	We discuss the polynomial setting first. More details can be found in
	the publications of \cite{q961,p512,Wa11}. Suppose an explicit QR step
	is executed on the Hessenberg matrix $H$, namely $\hat{H}=\hat{Q}^H H
	\hat{Q}$. This gives us the associated equation $H \hat{Q} = \hat{Q} \hat{H}$,
	which links, in the case of an unreduced Hessenberg $H$ one to one to
	a Krylov subspace $\mathcal{K}(H,q_1)$. More precisely the columns of
	$Q$ span the Krylov subspace and are identical to the one executing
	the similarity transformation linked to the explicit QR algorithm. We
	have (see \cite{q961})
	\begin{equation*}
		\mbox{span}\{q_1,\ldots,q_j\} = \mathcal{K}_j(H,q_1), \quad
		\mbox{for}\; j=1,\ldots,n.
	\end{equation*}
	As $q_1=(H-\tilde{z}I) e_1$, we get, for $j=1,\ldots,n$,
	\begin{eqnarray*}
		\mbox{span}\{q_1,\ldots,q_j\}
		& = & \mathcal{K}_j(H,(H-\tilde{z}I) e_1) \\
		& = &  (H-\tilde{z}I)\mathcal{K}_j(H, e_1)
		= (H-\tilde{z}I)\mbox{span}\,\{e_1,\ldots,e_j\}.
	\end{eqnarray*}
	By the unreduced Hessenberg structure of $H$, we can deduce that in
	case $\tilde{z}$ equals an eigenvalue of $H$, the eigenvector must be
	$q_n$. Considering the link between QR and RQ, we see that in the RQ
	algorithm the eigenvector must be positioned in the top row.
	
	A similar strategy to position a left and right eigenvector in the
	case of the RQZ algorithm can be followed. Details, theorems and
	proofs can be found in \cite{CaMeVa19}. The analogue of the explicit QR method for the RQZ
	method is based on rational Krylov subspaces. We have to consider two
	rational Krylov subspaces and the orthogonal matrices $Q$ and $Z$
	whose columns span these Krylov subspaces determine the equivalence
	transformation. Identical reasoning as for the polynomial case can be
	carried out: in case we use a shift equal to an eigenvalue left and
	right eigenvectors can be found in the last columns of $Q$ and $Z$ respectively.

	\section{Manipulating the generalized eigenvalue decomposition}\label{sec:GEVP}
	Suppose we have solved Problem~\ref{problem:HPIEP}, providing us with the
	decomposition $ZQK_m=QH_m$. In this section we will show how to modify
	this decomposition in order to solve Problem~\ref{problem:downdateHP}.
	To do so, we compute two unitary  matrices $R$ and $S$, that transform
	the pencil $(H_m,K_m)$ to a new pencil $(\hat{H},\hat{K})=(RH_mS^H,RK_mS^H)$. This leads to the equation
	\begin{equation*}
		Z (QR^H) (RK_mS^H) = (QR^H) (R H_m S^H), 
	\end{equation*}
	or similarly, after some diagonal permutations in $Z$ and row
	permutations\footnote{We remark that the permutations are only added
		to visualize the block structure of the matrices involved in $QR^H$} in $QR^H$:
	\begin{equation*}
		\begin{bmatrix}
			\tilde{z} \\
			&\tilde{Z}
		\end{bmatrix}
		\begin{bmatrix}
			1 \\
			&\tilde{Q}
		\end{bmatrix}
		\begin{bmatrix}
			\tilde{k} \\
			&\tilde{K}_{m-1}
		\end{bmatrix}
		=
		\begin{bmatrix}
			1 \\
			&\tilde{Q}
		\end{bmatrix} 
		\begin{bmatrix}
			\tilde{h} \\
			&\tilde{H}_{m-1}
		\end{bmatrix},
	\end{equation*}
	where $\tilde{h}/\tilde{k}=\tilde{z}$.
	Extracting the lower right $(m-1)\times (m-1)$ block out of the
	equation should provide us the desired solution to the downdating problem.
	Of course also the constraint on the weights still needs to be
	satisfied.
	
	We discuss two approaches to construct the transformations $R$ and
	$S$. An implicit method based on the structure of the Hessenberg
	pencil and a method based on the left and right eigenvector.
	We do not discuss an explicit method.

	\subsection{Hessenberg pencil and the rational RQ method}

	A proper Hessenberg matrix generalizes to a proper Hessenberg pencil
	in the sense that also here no breakdown will occur.  Mathematically
	we end up with the following definition.
	
	\begin{definition}[Proper Hessenberg pencil]\label{def:properHP}
		A Hessenberg pencil $(H,K)\in\mathbb{C}^{m\times m}\times
		\mathbb{C}^{m\times m}$ is called \emph{proper} if none of the
		subdiagonal elements are simultaneously zero, i.e., $\vert
		h_{\ell+1,\ell}\vert + \vert k_{\ell+1,\ell} \vert \neq 0$ for
		$\ell=1,2,\dots,m-1$. Also, the first columns of $H$ and $K$ must
		be linearly independent, as
		well as the last rows of $H$ and $K$.
	\end{definition}
	
	\begin{definition}[Normal Hessenberg pencil]
		A Hessenberg pencil $(H,K)$ is said to be normal, if $H$ and
		$K$ are simultaneously unitarily diagonalizable, i.e., there exists
		unitary $R$ and $S$ such that $R^H H S$ and $R^H K S$ are both
		diagonal.
	\end{definition}

	As in the polynomial case, there are several variants to implement  a
	perfectly shifted rational QZ step. The essential theorem to prove the equivalences links
	left and right eigenvectors to columns in the equivalence
	transformation; columns that are in fact eigenvectors.
	
	In the polynomial case, we could immediately, see Lemma~\ref{lemma:RQ_perfectShift}, extract
	the position of the eigenvector in the unitary matrix $Q$ that
	executes the similarity transformation. For an explicit rational RQ
	step, we needed to focus on rational Krylov spaces to construct the
	unitary matrices $R$ and $S$ that determine the equivalence
	transformation as discussed in the previous section. The eigenvector
	is required in both implementations of the rational QZ method, either
	via an implicit chasing or directly via the eigenvector.

	As a consequence we can formulate the following theorem, that proves
	deflation for a normal proper Hessenberg pencil.
	
	\begin{theorem}[Deflation for perfect shift backward RQZ]\label{theorem:perfectShiftRQZ}
		Let $(H,K)$ be an $m\times m$ proper normal Hessenberg pencil.
		Let $\tilde{z}\in\mathbb{C}$ be an eigenvalue of $(H,K)$
		with $\tilde{z}\notin \Xi$. Consider a right eigenvector $\hat{s}$,
		and a left eigenvector $\hat{r}$. Executing an equivalence
		transformation, with $\hat{R}$ and $\hat{S}$, where $\hat{R}e_1=\hat{r}$ and $\hat{S}e_1=\hat{s}$,
		results in a pencil $\hat{R}^H(H,K)\hat{S}$, isolating the eigenvalue
		$\tilde{z}$ in the upper left corner.
	\end{theorem}

	\begin{proof}
		
		Let $R,S$ be the unitary factors of diagonalizing the normal
		pencil $(H,K)$. That is $H= R Z_H S^H, K=R Z_K S^H$, and both
		$Z_H$ and $Z_K$ are diagonal. Assume $\hat{r}$ and $\hat{s}$
		link to eigenvalue $z_j$ in $(H,K)$, and denote the associated
		eigenvectors as ${r}_j$ and ${s}_j$. Let $\hat{R}$ and
		$\hat{S}$ be of the form:
		$\hat{R}:=\begin{bmatrix}
				\hat{r}^H\\
				\hat{R}_{m-1} 
			\end{bmatrix}$
			and 
			$\hat{S}:=\begin{bmatrix}
				\hat{s}^H\\
				\hat{S}_{m-1} 
			\end{bmatrix}$.
		
		By relying on the orthogonality of the eigenvectors we obtain
		\begin{align*}
			\hat{r}^H R &=
			\hat{r}^H
			\begin{bmatrix}
				\vert & & \vert & \vert & \vert & & \vert \\
				r_1 & \dots & r_{j-1} & r_j & r_{j+1} & \dots & r_{m}\\
				\vert & & \vert & \vert & \vert & & \vert
			\end{bmatrix} \\
			&= \begin{bmatrix}
				0 & \dots & 0 & \alpha &0 & \dots & 0\\
			\end{bmatrix}\\
			&= \alpha e_j^\top,
		\end{align*}
		where $\alpha = \hat{r}^H r_j \neq 0$. A similar statement can
		be made for $\hat{s}$ applied to $S$.		
		As a result we get
		\begin{equation*}
			\hat{R} H \hat{S}^H = \hat{R} R Z_H S^H \hat{S}^H 
			= \begin{bmatrix}
				\lambda_H & \boldsymbol{0}\\
				\boldsymbol{0} & \widetilde{Z}_{m-1}
			\end{bmatrix}.
		\end{equation*}
		A similar deduction can be made for the matrix $K$.
		As a result we obtain a new unreduced Hessenberg pencil
		$(\widetilde{H}_{m-1},\widetilde{K}_{m-1})$, with spectrum
		$\sigma(H,K)\backslash \{\lambda\}$.
	\end{proof}
	
	By construction, in the actual algorithms we use unitary Hessenberg
	matrices for $\hat{R}$ and $\hat{S}$, the Hessenberg pencil ($\hat{R},\hat{S}$) must be unreduced.
	Let us formulate a theorem analoguous to
	Theorem~\ref{theorem:RQ_equivalent} for the pencil case.
	\begin{theorem}
		Let $(H,K)\in\mathbb{C}^{m\times m}$ be a proper Hessenberg pencil and $\lambda \in \sigma(H)$.
		Then the following statements hold:
		\begin{enumerate}
			\item $(H,K)$ has a normalized left and right eigenvector $\hat{r}$ and $\hat{s}$ corresponding to $\lambda$
			\begin{equation*}
				H \hat{s} = \lambda K \hat{s},\quad \Vert \hat{s} \Vert_2 =
				1, \quad
				\hat{r}^H H = \lambda   \hat{r}^H  K,\quad \Vert \hat{r} \Vert_2 = 1,
			\end{equation*}
			which are unique up to unimodular scaling and $e_m^\top
			\hat{s} \neq 0$, as well as $e_m^\top
			\hat{r} \neq 0$.
			\item Essentially unique sequences of
			$\{R_i\}_{i=1}^{m-1}$ and $\{S_i\}_{i=1}^{m-1}$ with  $R_i,S_i \in\mathfrak{C}_i$ forming the matrix
			\begin{equation*}
				\hat{R} := R_1 R_2 \dots R_{m-1}
				\mbox{ and }
				\hat{S} := S_1 S_2 \dots S_{m-1}
			\end{equation*}
			exist such that $(\hat{H},\hat{K})=\hat{R} (H,K)
			\hat{S}^H$ is a proper Hessenberg pencil and
			$\hat{R}e_1=\hat{r}$ as well as $\hat{S}e_1=\hat{s}$.
		\end{enumerate}
	\end{theorem}

	\section{Numerical algorithms for downdating rational functions}\label{sec:NumAlg:RF}
	For rational functions with prescribed poles the structured matrix in the IEP is a Hessenberg pencil $(H_m,K_m)\in\mathbb{C}^{m\times m}$.
	The corresponding IEP is given in Problem \ref{problem:HPIEP}.
	In this problem, the ratio restriction on subdiagonal elements of the
	Hessenberg pencil $\frac{h_{i+1,i}}{k_{i+1,i}} = \xi_i\in \overline{\mathbb{C}}$, $i=1,2,\dots, m-1$, provides the connection to rational functions with poles $\Xi = \{\xi_1,\xi_2,\dots, \xi_{m-1} \}$.
	For brevity, we will call the $i$th subdiagonal element of $(H,K)$
	both $h_{i+1,i}$ and $k_{i+1,i}$ simultaneously, we also name this the $i$th \emph{pole position} of the pencil.
	And we will say that $\xi$ appears on pole position $i$ if $\frac{h_{i+1,i}}{k_{i+1,i}} = \xi$.\\
	The procedures from Section \ref{sec:down:poly} can be generalized if, instead of a RQ
	step, we use the RQZ algorithm of \cite{CaMeVa19}.
	
	For the rational case, however, we will not have an explicit matrix
	method. Constructing an orthogonal basis for the Krylov subspace in
	the polynomial setting is easy because of the close connection between
	the Krylov subspace and the columns of the Hessenberg matrix. For the
	rational setting it is computationally and numerically not advisable
	to construct the rational Krylov subspace, whose basis is explicitly
	required to execute an explicit RQZ step. Hence we will focus solely
	on the implicit matrix method and on the eigenvector method.

	\subsection{The implicit matrix method}
	
	First we introduce the RQZ step where we have to chase in the other
	direction  \cite{CaMeVa19}, which
	consists of essentially two operations. The first operation is
	\emph{pole swapping} and takes place in the middle of the pencil
	$(H_m,K_m)$. The second operation is modifying the first or last pole by any given pole.
	Both procedures work essentially on $2\times 2$ matrices, hence we
	will introduce them for $2\times 2$ matrices and identify the relevant
	submatrices when describing the whole procedure.
	
	Swapping poles interchanges the poles on two neighboring pole positions, i.e., $\xi_i$ appearing on pole position $i$ and $\xi_{i+1}$ on position $i+1$ are interchanged such that $\xi_i$ now appears on pole position $i+1$ and $\xi_{i+1}$ on position $i$.
	Lemma \ref{lemma:poleSwapping} shows that swapping poles can be done
	using only unitary $2\times 2$ transformations.
	\begin{lemma}[Pole swapping \cite{CaMeVa19,BeGu15}]\label{lemma:poleSwapping}
		Let ${H} = \begin{bmatrix}
			h_{1,1} & h_{1,2}\\
			& h_{2,2}
		\end{bmatrix}$ and ${K} = \begin{bmatrix}
			k_{1,1} & k_{1,2}\\
			& k_{2,2}
		\end{bmatrix}$ and set $\xi_1 = \frac{h_{1,1}}{k_{1,1}}$ and $\xi_2 = \frac{h_{2,2}}{k_{2,2}}$.
		Then $R_1,S_1\in\mathfrak{C}_1$ can be constructed such that
		\begin{align*}
			R_1{H} S_1^H &= R_1 \begin{bmatrix}
				h_{1,1} & h_{1,2}\\
				& h_{2,2}
			\end{bmatrix} S_1^H = \begin{bmatrix}
				\tilde{h}_{1,1} & \tilde{h}_{1,2}\\
				& \tilde{h}_{2,2}
			\end{bmatrix},\\
			R_1{K} S_1^H &= R_1 \begin{bmatrix}
				k_{1,1} & k_{1,2}\\
				& k_{2,2}
			\end{bmatrix} S_1^H = \begin{bmatrix}
				\tilde{k}_{1,1} & \tilde{k}_{1,2}\\
				& \tilde{k}_{2,2}
			\end{bmatrix},
		\end{align*}
		where $\frac{\tilde{h}_{2,2}}{\tilde{k}_{2,2}}=\xi_1 $ and $\frac{\tilde{h}_{1,1}}{\tilde{k}_{1,1}}=\xi_2$, i.e., the poles are swapped on the diagonal of the matrices.
	\end{lemma}
	
	In order to apply Lemma \ref{lemma:poleSwapping} to a Hessenberg
	pencil, the plane rotations of size $2\times 2$ must be embedded in an
	identity matrix such that they act on the correct submatrix. As a
	result, particular poles on the subdiagonal will be swapped.
	
	Changing a pole is possible on the first pole position or the last, i.e., pole position $1$ or $m-1$.
	Lemma \ref{lemma:poleChanging} states that changing a pole can be
	performed by a unitary similarity transformation.
	We formulate the lemma only for the last pole, the first pole can be
	altered similarly.
	\begin{lemma}[Changing the last pole, \cite{CaMeVa19,BeGu15}] \label{lemma:poleChanging}
		Let $(H,K)\in\mathbb{C}^m$ be a proper Hessenberg pencil with
		poles $\Xi = \{\xi_1,\xi_2,\dots, \xi_{m-1} \}$, where
		$\xi_i\in \overline{\mathbb{C}}$ appears on pole position $i$.
		Let $\tilde{\xi}_{m-1}= \frac{\tilde{h}}{\tilde{k}}\notin \sigma(H,K)$ be a given pole and let $\xi_{m-1} = \frac{h}{k}$.
		
		Construct a unitary transformation $S\in\mathfrak{C}_1$ such that\footnote{
			Note that in case $k$ and $\tilde{k}$ are nonzero that
			$x^T = \tilde{\gamma} e_{m}^T ( H - \tilde{\xi}_1 K)^{-1}(H - \xi_1 K) $, for a
			scaling factor $\tilde{\gamma}$.}
		
		\begin{equation*}
			x^T S^H = \alpha e_{m}^T, \quad\mbox{ for }\quad x^T
			= \tilde{\gamma} e_{m}^T ( \tilde{k} H - \tilde{h} K)^{-1}(k H - h K) ,
		\end{equation*}
		for a constant $\gamma$. Then the pencil $(HS^H,KS^H)$ has $\tilde{\xi}_{m-1}$ on the last pole position and the poles on position 1 until $m-2$ equal those of $(H,K)$ on corresponding positions.
		
	\end{lemma}
	Let $Z = \textrm{diag}(z_1,\dots,z_m)$, $w =\begin{bmatrix}
		w_1 & \dots & w_m
	\end{bmatrix}^\top$ and $\Xi = \{\xi_1,\dots,\xi_{m-1} \}$, and suppose we
	have a solution to the HPIEP $(H_m,K_m)$, such that $ZQ_m K_m =
	Q_mH_m$, with $Q_m$ unitary and the poles appear on the subdiagonal of
	the pencil.
	Now, we would like to downdate this solution: We want to remove a node $\tilde{z}$.
	
	\begin{enumerate}
		\item The implicit RQZ method, where we chase backwards is initialized
		by replacing the last pole, i.e., the pole in the $(m-1)$st position
		on the subdiagonal with $\tilde{z}=h/k$ as described in Lemma~\ref{lemma:poleChanging}
		
		\item Next we keep swapping the pole to move it slowly to the top of
		matrix. We use the unitary transformations from Lemma~\ref{lemma:poleSwapping}.  
	\end{enumerate}
	At the very end, in the exact theoretical setting the resulting pencil $(\widetilde{H},\widetilde{K})$ has structure
	\begin{equation*}
		\left(	\begin{bmatrix}
			{h} \\
			& \times & \times & \times & \dots & \times & \times\\
			& \times & \times & \times & \dots & \times & \times\\
			& 		 & \ddots & \vdots & 	   & \vdots & \vdots\\
			& 		 & 		  & \times & \dots & \times & \times\\
			& 		 & 		  & 	   &\ddots & \vdots & \vdots\\
			& 		 & 		  & 	   & 	   & \times & \times
		\end{bmatrix} , \begin{bmatrix}
			k \\
			& \times & \times & \times & \dots & \times & \times\\
			& \times & \times & \times & \dots & \times & \times\\
			& 		 & \ddots & \vdots & 	   & \vdots & \vdots\\
			& 		 & 		  & \times & \dots & \times & \times\\
			& 		 & 		  & 	   &\ddots & \vdots & \vdots\\
			& 		 & 		  & 	   & 	   & \times & \times
		\end{bmatrix}\right),
	\end{equation*}
	where $\frac{h}{k} = \tilde{z}$ equals the node to be downdated.
	
	A closer look at the algorithm reveals that the pole that will be
	deflated will always be the pole in the last position, i.e., the
	$(m-1)$st position on the subdiagonal as it will be replaced by
	$\tilde{z}$. If one, however, wishes to keep that pole and remove
	another pole, one can make use of Lemma~\ref{lemma:poleSwapping} to move the pole that one wants to get rid of to position $(m-1)$, after which one starts the implicit algorithm.
	Suppose that a single implicit step does not provide sufficient accuracy,
	meaning that one can not immediately deflate the desired node. In that
	case a second chasing step can be considered, similarly as in the
	polynomial setting. As a consequence the vector of weights will not be
	correct anymore and structure restoring transformations are required.
	
	\subsection{The eigenvector method}\label{sec:downdateRF_MaVD}
	The generalization of the eigenvector method to downdating of a
	Hessenberg pencil $(H_m,K_m)$ requires us to compute both the left and
	right eigenvector. 
	We follow Theorem~\ref{theorem:perfectShiftRQZ} and employ the same three steps as described for the Hessenberg matrix case in Section~\ref{sec:downdatePoly_MaVD}.
	Assume we have a number $\tilde{z}$ close enough to an eigenvalue of the pencil $(H_m, K_m)$.
	First we perform all three steps for the left eigenvector and afterwards for the right eigenvector.
	In the first step the left eigenvector $r$ is obtained by evaluating the sequence of orthogonal rational functions generated by $(H_m,K_m)$ in the given eigenvalue $\tilde{z}$.
	This eigenvector should satisfy
	\begin{equation}\label{eq:leftevec_quality}
		\Vert r^H \left(H_m-\tilde{z}K_m \right)\Vert_2 \approx \epsilon_{\textrm{mach}} \Vert H_m-\tilde{z} K_m\Vert_2,
	\end{equation}
	if this is not satisfied, iterative refinement can be used to improve the accuracy of $r$.
	The second step consists of computing a left eigenvector $\dot{r}$, starting from $r$, that is especially accurate in its trailing entries, i.e., 
	\begin{equation}\label{eq:cond_47_left}
		\left\Vert \begin{bmatrix}\hat{\epsilon}_1 & \hat{\epsilon}_2 & \dots &\hat{\epsilon}_m		\end{bmatrix}\right\Vert_2 \leq  \epsilon_{\textrm{mach}} \min\left(\Vert H_m\Vert_\textrm{F},\Vert K_{m}\Vert_\textrm{F} \right),
	\end{equation}
	with  $\hat{\epsilon}_i = \frac{e_i^\top (\dot{r}^H(H_m-\tilde{z}K_m))}{\left\Vert \begin{bmatrix}
			\dot{r}_{i-1} & \dot{r}_i & \dots & \dot{r}_{m}
		\end{bmatrix} \right\Vert_2}$ and $\dot{r}_{0} := 0$.	
	This vector is obtained in the same way as for the Hessenberg matrix case, making straightforward adjustments to take into account that now we deal with a pencil.
	In the third step a unitary similarity transformation is executed with the matrix $R$ that reduces $\dot{r}$ to $R^H \dot{r} = e_1$.
	The resulting pencil is $\left(R^H H_m, R^H K_m\right)$.\\
	The same three steps are repeated for the right eigenvector $s$, but now applied to the pencil $\left(R^H H_m, R^H K_m\right)$.
	The right eigenvector $s$ must be computed in a different way than the left eigenvector, because there is no connection to the sequence of orthogonal rational functions.
	Any method can be applied as long as the resulting eigenvector satisfies
	\begin{equation}\label{eq:rightevec_quality}
		\Vert \left(R^H H_m -\tilde{z} R^H K_m\right) s\Vert_2 \approx \epsilon_{\textrm{mach}} \Vert R^H H_m -\tilde{z} R^H K_m\Vert_2.
	\end{equation}
	After computing $\dot{s}$ satisfying, for  $\hat{\epsilon}_i = \frac{e_i^\top ((R^H H_{m}-\tilde{z}R^H K_{m}) \dot{s})}{\left\Vert \begin{bmatrix}
			\dot{s}_{i-1} & \dot{s}_i & \dots & \dot{s}_{m}
		\end{bmatrix} \right\Vert_2}$ and $\dot{s}_{0} := 0$,
	\begin{equation}\label{eq:cond_47_right}
		\left\Vert \begin{bmatrix}\hat{\epsilon}_1 & \hat{\epsilon}_2 & \dots &\hat{\epsilon}_m		\end{bmatrix}\right\Vert_2 \leq  \epsilon_{\textrm{mach}} \min\left(\Vert R^H H_m\Vert_\textrm{F},\Vert R^H K_m-\Vert_\textrm{F} \right),
	\end{equation}
	the matrix $S$ reducing $\dot{s}$ to $S^H \dot{s}=e_1$ is computed.
	The corresponding unitary similarity transformation on the pencil leads to $(R^H H_m S,R^H K_m S)$, which according to Theorem \ref{theorem:perfectShiftRQZ}, allows deflation of $\tilde{z}$.
	Also here a specific pole will be deflated. If deflation of another pole is
	desired, it needs to be brought to the bottom right corner first.
	Note that the choice to first treat the left eigenvector and afterwards the right eigenvector is only one possible way to implement this method.
	Determining the most stable variant is subject of future research and is out of the scope of this paper.

	\section{Numerical experiments}\label{sec:NumExp:RF}
	Two numerical experiments are performed to compare the implicit matrix and eigenvector method for solving Problem \ref{problem:downdateHP}.
	The first experiment uses the same setup as in Section \ref{sec:down:poly} and the second uses a sliding window.
	In both methods the unitary matrix $\tilde{Q}$ will be computed only to be used for the computation of these metrics, it is not used in the methods themselves. \\
	The metrics for the orthogonality error \eqref{eq:err_o} and weight error \eqref{eq:err_w} for the polynomial case can be immediately applied to the rational case.
	For the node error \eqref{eq:err_nodes} the eigenvalues $\lambda_j$ are now the eigenvalues of a pencil $(\tilde{H}_{m-k},\tilde{K}_{m-k})$.
	Additionally we use the following two metrics, where $\tilde{Z}$ denotes the appropriately redefined matrix of nodes,
	\begin{enumerate}
		\item The \textit{recurrence error} for the pencil $(\tilde{H}_{m-k},\tilde{K}_{m-k})$:
		\begin{equation}\label{eq:err_r_RF}
			\textrm{err}_{\textrm{r}} := \frac{\Vert \tilde{Z} \tilde{Q}_{m-k}\tilde{K}_{m-k} - \tilde{Q}_{m-k} \tilde{H}_{m-k}\Vert _2}{\max\left(\Vert \tilde{Z}\tilde{Q}_{m-k} \tilde{K}_{m-k} \Vert_2, \Vert \tilde{Q}_{m-k} \tilde{H}_{m-k} \Vert_2 \right)}.
		\end{equation}
		\item The \textit{pole error} quantifies the accuracy of the poles appearing in the pencil $(\tilde{H}_{m-k},\tilde{K}_{m-k})$ to the given poles $\{\tilde{\xi}_j\}_{j=1}^{m-k-1}$.
		Let $\tilde{h}_{i,j}$ and $\tilde{k}_{i,j}$ denote the $(i,j)$th element of $\tilde{H}_{m-k}$ and $\tilde{K}_{m-k}$, respectively, then the pole error is defined as
		\begin{equation}\label{eq:err_poles}
			\textrm{err}_\textrm{p} = \max_j \left(\begin{cases}
				\left\vert\frac{\frac{\tilde{h}_{j+1,j}}{\tilde{k}_{j+1,j}}-\tilde{\xi}_j}{\tilde{\xi}_j}\right\vert, \quad &\text{if } \tilde{\xi}_j \in \mathbb{C},\\
				\vert \frac{\tilde{k}_{j+1,j}}{\tilde{h}_{j+1,j}}\vert, \quad &\text{if } \tilde{\xi}_j,\text{ if infinite}
			\end{cases}\right).
		\end{equation}
	\end{enumerate}
	For the eigenvector method we will track the two conditions for the left eigenvector \eqref{eq:leftevec_quality} and $\eqref{eq:cond_47_left}$ and for the right eigenvector \eqref{eq:rightevec_quality} and \eqref{eq:cond_47_right}.
	In both experiments we start from a pencil $(H_m,K_m)$, the solution to Problem~ \ref{problem:HPIEP}, obtained by the updating procedure proposed by \cite{VBVBVa22}.
	
	\subsection{Unit circle}
	The $m$ nodes $\{z_j\}_{j=1}^m$ are chosen in a balanced way on the unit circle, as described in Section \ref{sec:exp:polyUC} and the weights are $w = \frac{1}{\sqrt{m}}\begin{bmatrix}
		1 & 1 & \dots & 1
	\end{bmatrix}^\top$.
	The poles are chosen on a smaller circle with radius $1-\delta$ and on a bigger circle with radius $1+\delta$, with $0<\delta <1$.
	They appear in pairs, one pole on the smaller circle under a certain angle and one on the bigger circle with the same angle.
	These pairs are added in the same balanced order as the nodes.
	\begin{figure}[!ht]
		\centering
		\includegraphics{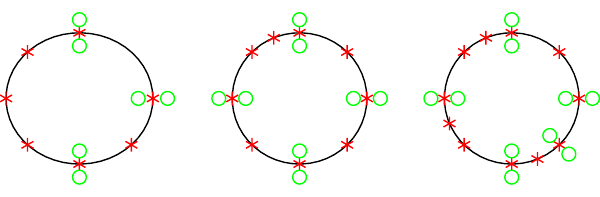}
		\caption{Balanced order of choosing $m$ equidistant nodes ${\color{red} \ast}$ on a circle with radius 1 and $m-1$ poles ${\color{green} \circ}$ appearing in pairs on circles with radius $(1-\delta)$ and $(1+\delta)$ for $m=7,9$ and $11$.}
		\label{fig:balancedCircle}
	\end{figure}
	The experiment downdates half of the nodes in the reverse order in which they are added.
	Figure \ref{fig:UC_opt_m201_RF} shows the metrics for $\delta = 0.1$, $m=201$ and $n_\textrm{IR}=1$.
	Both methods perform satisfactory, for the eigenvector method the metrics show a faster deterioration of the solutions $(\tilde{H}_{m-k},\tilde{K}_{m-k})$ than for the matrix method.
	\begin{figure}[!ht]
		\centering
		\includegraphics{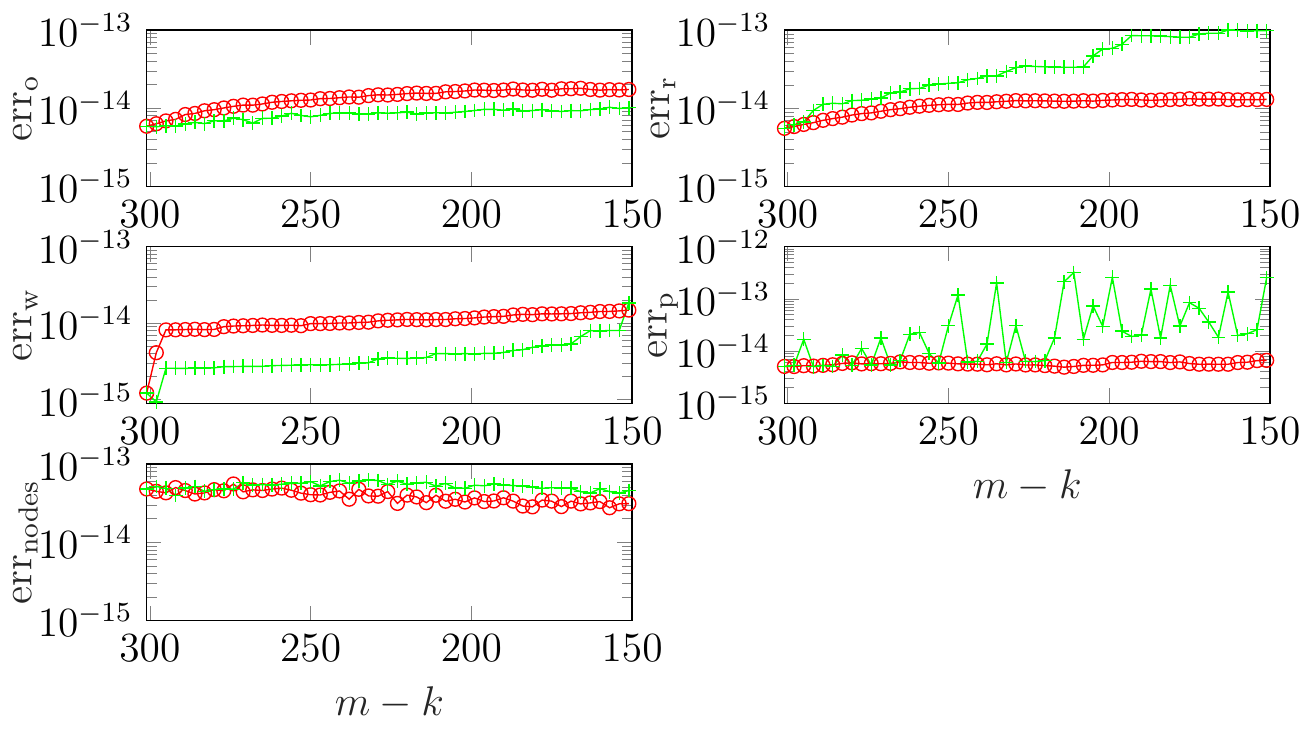}
		\caption{Metrics comparing the matrix method ${\color{red}\circ}$ and eigenvector method $\color{green}+$ for the unit circle experiment with $\delta = 0.1$, $m=201$ and $n_\textrm{IR}=1$.}
		\label{fig:UC_opt_m201_RF}
	\end{figure}

	\subsection{Real line}\label{sec:exp:RF_realLine}
	For the following experiment with nodes on the real line we change the setup of the experiment, we choose to employ a sliding window.
	We start from a solution $(H_{m},K_{m})$, $Q$ to Problem \ref{problem:HPIEP} for $\{z_j\}_{j=1}^{m}, \{\xi_j\}_{j=1}^{m-1}$.
	Then downdate the first two nodes and poles, i.e., by solving Problem \ref{problem:downdateHP} twice resulting in a intermediate solution to the IEP with $\{z_j\}_{j=3}^{m}, \{\xi_j\}_{j=3}^{m-1}$.
	Immediately after this, two new nodes and poles are updated, resulting in the solution $(\tilde{H}^{(1)}_{m},\tilde{K}^{(1)}_{m})$, $\tilde{Q}^{(1)}$ for $\{z_j\}_{j=3}^{m+2}, \{\xi_j\}_{j=3}^{m+1}$.
	This is shown in Figure \ref{fig:window_realLine} for equidistant nodes $\{x_j\}_{j=1}^m$ in the interval $\left[a,b\right]$ and poles
	\begin{equation}\label{eq:poles}
		\xi_j = \begin{cases}
			a + (j-1/2) \Delta x + \delta \imath,\quad \text{if } j \text{ odd},\\
			a + (j-3/2) \Delta x - \delta \imath,\quad \text{if } j \text{ even}.
		\end{cases}
	\end{equation}
	This process is repeated $\ell$ times, resulting at the $k$th time in the solution $(\tilde{H}^{(k)}_{m},\tilde{K}^{(k)}_{m})$, $\tilde{Q}^{(k)}$ of size $m\times m$.
	Instead of computing the metrics for $(\tilde{H}_{m-k},\tilde{K}_{m-k})$ which decrease in size as $k$ increases, we now compute the metrics for $(\tilde{H}^{(k)}_m,\tilde{K}^{(k)}_m)$ of size $m\times m$ for several values of $k$.
	\begin{figure}[!ht]
		\centering
		\includegraphics{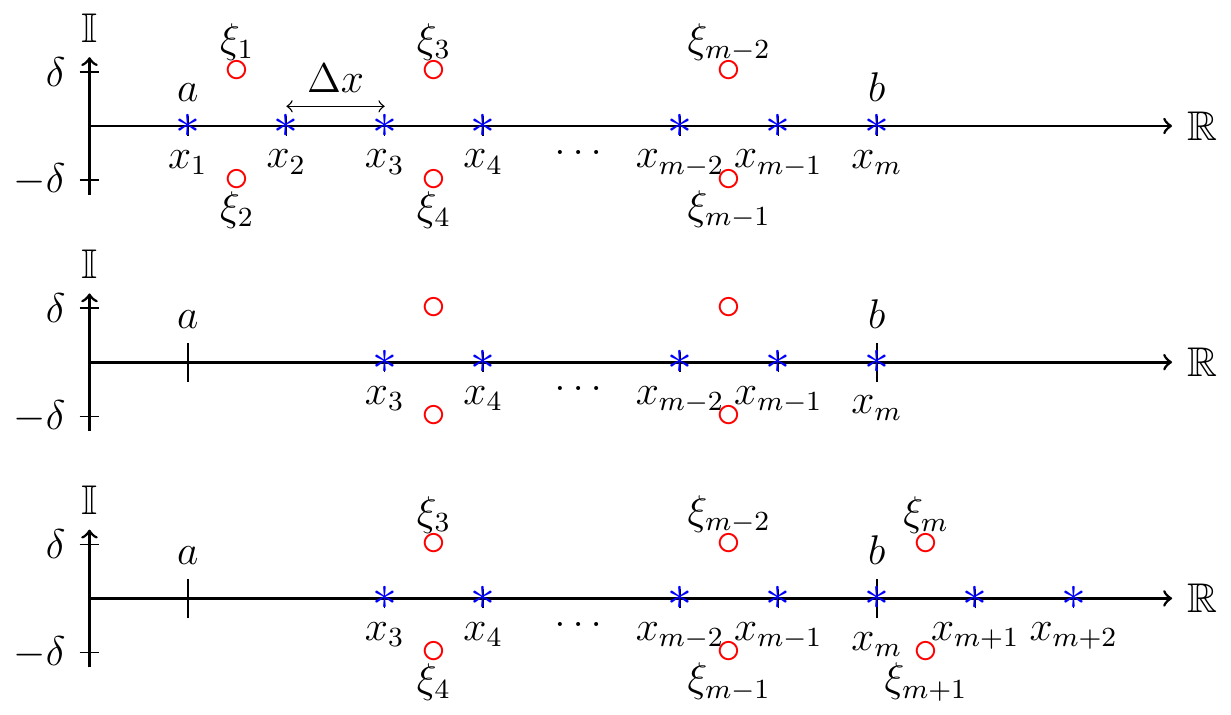}
		\caption{A single step in the experiment on the real line is illustrated in the complex plane. The top figure shows the initial setup, $m$ equidistant nodes $\{x_j\}_{j=1}^m$ ${\color{blue}\ast}$ in $\left[a,b\right]$ and $m-1$ poles $\{\xi_j\}_{j=1}^{m-1}$ \eqref{eq:poles}.
			The middle figure shows the setup for downdating: remove the first two nodes and first two poles. 
			The bottom figure shows setup for updating: add two equidistant nodes $x_{m+1} = b+\Delta x$ and $x_{m+2} = b+2\Delta x$ and poles $\xi_{m}$ and $\xi_{m+1}$ according to formula \eqref{eq:poles}.}
		\label{fig:window_realLine}
	\end{figure}
	
	In Figure \ref{fig:exp:window_realLine} the metrics for $m=201$, $\left[a,b\right] = \left[0,2\pi\right]$, $\delta = 0.1$ and $\ell=100$ are shown.
	The implicit matrix method performs very well, with only a small error growth.
	The eigenvector method, with $n_\textrm{IR}=1$, shows a faster error growth.
	To explain this we take a look at the quantities shown in Figure \ref{fig:exp:window_realLine_emethod}.
	As $k$ increases (we slide further) these quantities indicate that the accuracy of the eigenvectors $\hat{r}$ and $\hat{s}$ deteriorates steadily, which leads to an inaccurate $\dot{r}$ and $\dot{s}$ and therefore a deterioration in the quality of the computed pencil.
	For larger values of $n_\textrm{IR}$ the results do not improve, hinting at another cause for the fast decrease in accuracy, e.g., the conditioning.
	Exploring this further is out of the scope of this paper.
	\begin{figure}[!ht]
		\centering
		\includegraphics{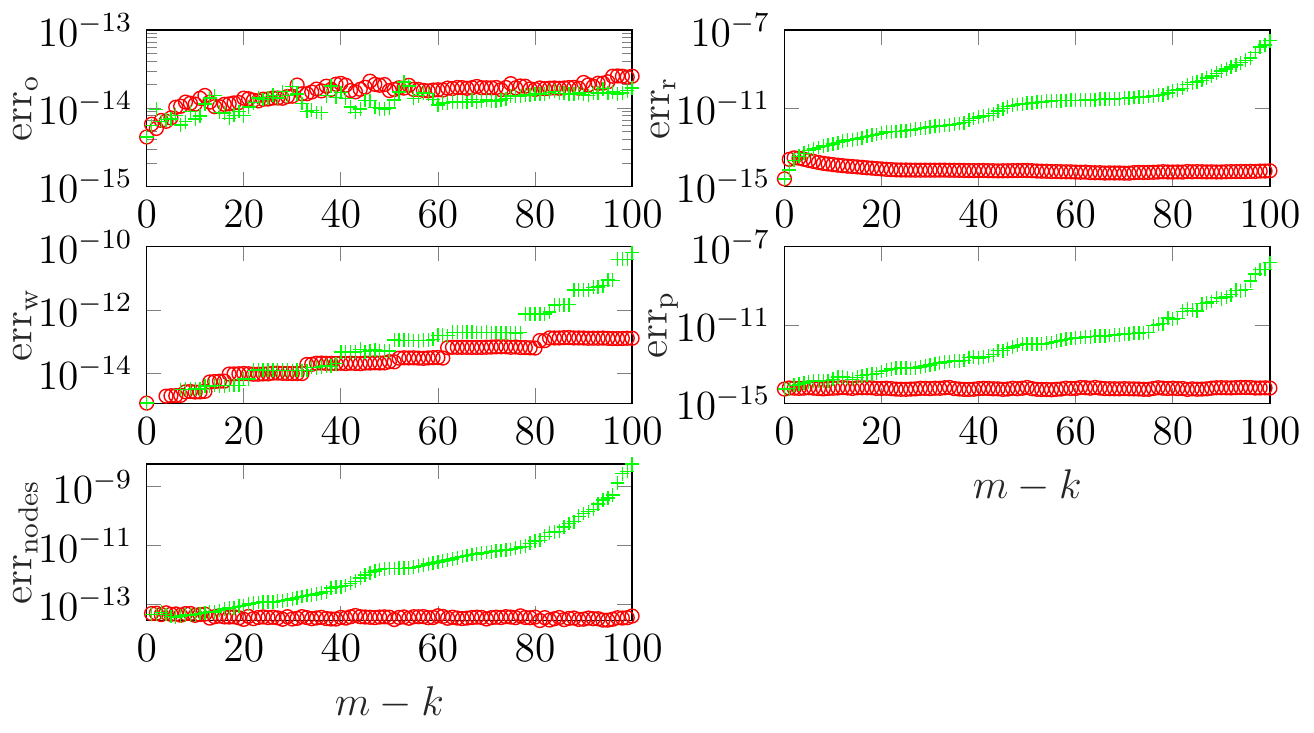}
		\caption{Metrics for the sliding window experiment on the real line. The initial set of nodes is a set of $m=201$ equidistant nodes in $\left[0,\pi\right]$ and $\delta=0.1$ for the poles. The window is slid $\ell = 100$ times as described in Figure \ref{fig:window_realLine}.}
		\label{fig:exp:window_realLine}
	\end{figure}
	
	\begin{figure}[!ht]
		\centering
		\includegraphics{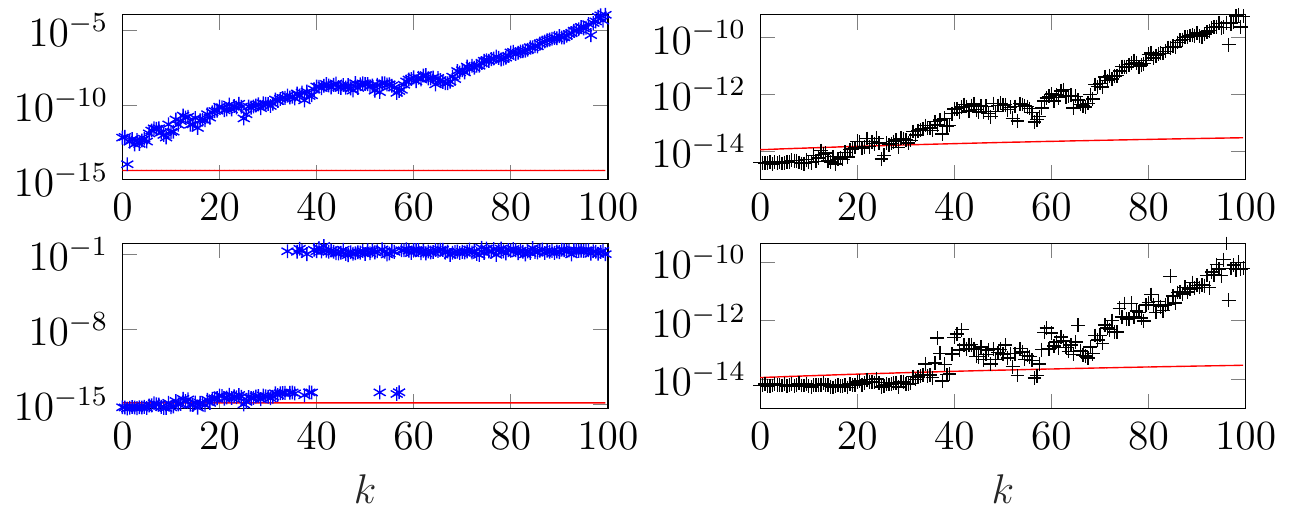}
		\caption{Quantities for the eigenvector method applied to the sliding window experiment on the real line, with $m=201$ equidistant nodes in $\left[0,\pi\right]$, $\delta=0.1$ and $\ell = 100$.
			Top left: \eqref{eq:leftevec_quality}, top right: \eqref{eq:cond_47_left}, bottom left: \eqref{eq:rightevec_quality} and bottom right: \eqref{eq:cond_47_right}. The left hand side of these inequalities are indicated by ${\color{blue} \ast}$ and the right hand side by ${\color{red} -}$.}
		\label{fig:exp:window_realLine_emethod}
	\end{figure}

	\newpage
	\section{Sliding window scheme}\label{sec:sliding_window}
	An important application of the downdating methods proposed in this paper combined with updating methods, e.g., those by \cite{m275,VBVBVa22}, is a sliding window scheme for data or function approximation.
	Consider an unknown function of interest $f(x)$ which is known only in a set of nodes $\{z_j\}_{j=1}^{m}$ with associated weights $\{w_j\}_{j=1}^m$, i.e., we have the data $\{w_j,z_j,f(z_j)\}_{j=1}^m$.
	From this data we would like to obtain an approximation $g(x)$ in the space of polynomials $\mathcal{P}_n$ or space of rational functions $\mathcal{R}_n^\Xi$, of restricted degree $n<m$, which minimizes the least squares criterion
	\begin{equation*}
		\sum_{j=1}^m \vert w_j\vert^2 \vert g(z_j)-f_j\vert^2.
	\end{equation*}
	\cite{Fo57} concluded that the best way to solve a least squares problem in $\mathcal{P}_n$ (or $\mathcal{R}_n^{\Xi}$) on a computer is by its representation in a basis of polynomials (or rational functions) orthogonal with respect to the given data.
	That is, the inner product $\langle p,q\rangle := \sum_{k=1}^{m} \vert w_k\vert^2 p(z_j) \overline{q(z_j)}$.
	Let $\{r_0(z),r_1(z),\dots, r_{n-1}(z)\}$ denote a sequence of rational functions (or polynomials) orthogonal with respect to this inner product.
	Their recurrence coefficients can be obtained by solving Problem~\ref{problem:HPIEP} (or Problem \ref{problem:HIEP}), as well as the orthonormal basis $Q_{n}\in\mathbb{C}^{m\times n}$ for $\mathcal{R}_n^{\Xi}$.
	Then the solution $g(z)\in\textrm{span}\{r_0(z),r_1(z),\dots, r_{n-1}(z)\} = \mathcal{R}_n^\Xi$ is
	\begin{equation*}
		g(z) = \sum_{d=0}^{n-1} \alpha_d r_d(z), \quad \text{with } \begin{bmatrix}
			\alpha_0\\
			\alpha_1\\
			\vdots\\
			\alpha_{n-1}
		\end{bmatrix}
		= Q_n^H \begin{bmatrix}
			w_1 f_1\\
			w_2 f_2\\
			\vdots\\
			w_m f_m
		\end{bmatrix}.
	\end{equation*}
	Note that $Q_n$ does not have to be formed explicitly, the transformations forming $Q_n$ can be directly applied to the vector with data.
	The error of such an approximation is quantified by $\Vert f(x) - g(x) \Vert_\infty$, which is estimated by computing these functions on an interval with $10m$ equidistant nodes.
	
	We illustrate the sliding window scheme for the function $f(x) = (\cos^2 x +1)^{-1}$.
	This function has singularities $\varphi_j = \arccos(\pm \imath) + j\pi$, $j\in\mathbb{Z}$, which suggests the use of a subset of these as poles for the rational function space.
	Suppose we possess the function values $\{f_j\}_{j=1}^m$ at $m=201$ equidistant points $\{x_j\}_{j=1}^m$ in the interval $\left[0,\alpha \pi\right]$ with equal weights $w_j=1$ for all $j$.
	Two least squares approximants are computed, $p(x)\in\mathcal{P}_{65}$ and $r(x)\in \mathcal{R}_{25}^\Xi$, with $\Xi = \{\varphi_j\}_{j=-6}^{2\alpha +5} \cup \{\infty\}_{j=1}^{12-2\alpha}$.
	This choice of poles corresponds to taking all singularities with real part $\textrm{Re}(\varphi_j)$ in the interval $\left[0,\alpha \pi\right]$ and the three first pairs on the left and right of this interval.
	The remaining poles are chosen to be infinity (i.e., polynomials).
	For $\alpha = 1$, both approximants have a high accuracy in the whole interval, the errors are shown in Figure \ref{fig:slid:alpha1_error}. 
	Note that the polynomial approximant has peaks near the endpoints of the interval, this is expected of polynomial interpolation in equidistant points.\\
	\begin{figure}[!ht]
		\centering
		\includegraphics{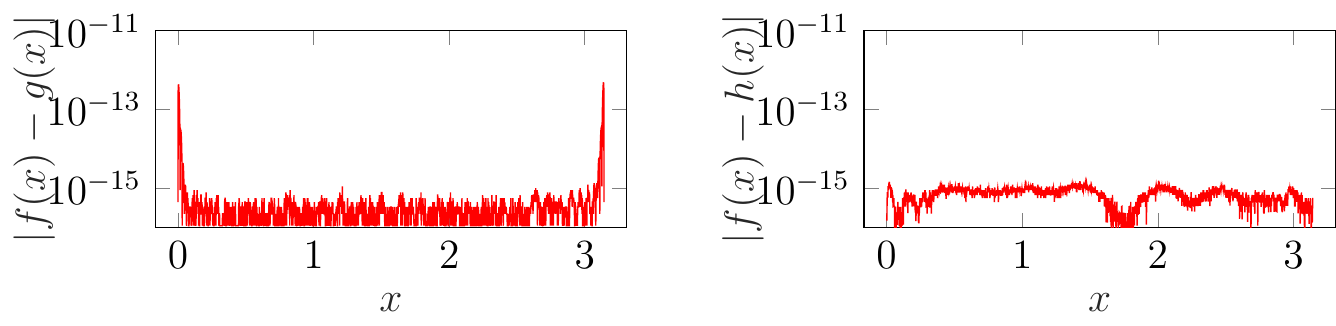}
		\caption{Error for least squares approximants $g(x)\in\mathcal{P}_{65}$ and $h(x)\in\mathcal{R}_{25}^\Xi$ to $f(x) = (\cos^2 x +1)^{-1}$ using $m=201$ equidistant nodes in $\left[0,\pi\right]$.}
		\label{fig:slid:alpha1_error}
	\end{figure}
	
	Starting from these approximants we compute approximants which slide over the interval $\left[0,\pi\right]$ to $\left[\pi,2\pi\right]$ in $\ell = 100$ equidistant steps.
	In each step $k=1,2,\dots,\ell$ the first two nodes $x_{2(k-1)}$ and $x_{2(k-1)+1}$ are downdated and two new nodes are updated on the right of the interval, in the same sense as described in Figure \ref{fig:window_realLine}.
	When $\textrm{Re}(\xi_j) \in \left[x_{(2(k-1))},x_{(2(k-1)+1)}\right]$ for some $j$, then the first pair of poles is replaced by a new pair of poles on the right, in the same spirit as the setup in Figure \ref{fig:window_realLine}.
	The metrics for this experiment are shown in Figure \ref{fig:slid:alpha1_metrics}, the metrics show that the up-and downdating is done satisfactorily.
	After an initial increase by a factor 100 from $k=0$ to $k=1$ the metric $\textrm{err}_\textrm{f}$ for the rational function increases only slightly, by a factor 10 overall.
	For the polynomial approximation there is only a factor 5 increase for this metric.\\
	\begin{figure}[!ht]
		\centering
		\includegraphics{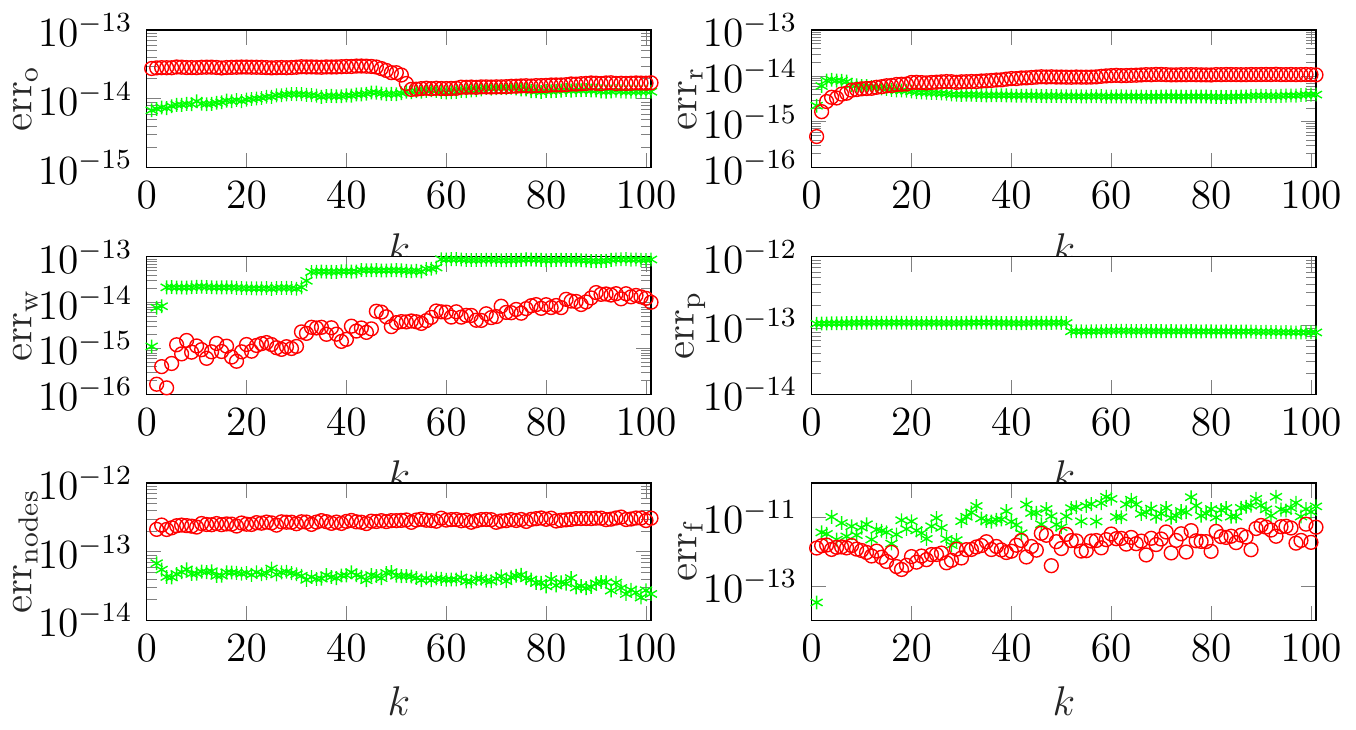}
		\caption{Metrics for sliding window of $m=201$ equidistant nodes on $\left[0,\pi\right]$ to $m$ equidistant nodes on $\left[\pi,2\pi\right]$. Metric $\textrm{err}_\textrm{f}$ for least squares approximants $g(x)\in\mathcal{P}_{65}$ and $h\in\mathcal{R}_{25}^\Xi$ to $f(x) = (\cos^2 x +1)^{-1}$ in this sliding window.
			For the polynomial space with the eigenvector method as downdating procedure ${\color{red} \circ}$ and the rational function space with the implicit matrix method as downdating procedure ${\color{green}\ast}$.}
		\label{fig:slid:alpha1_metrics}
	\end{figure}
	
	The power of rational functions becomes clear when the function $f(x)$ has more singularities near the interval of interest.
	For $\alpha=6$ there are 5 more pairs of singularities which strongly influence the approximation of $f(x)$.
	Without changing the size of the spaces we search for approximations to $f(x)$ on this larger interval.
	For the approximation the polynomial space remains the same, $\mathcal{P}_{65}$, and the rational function space remains of the same dimension $25$, but now there are 24 poles instead of 14 poles for $\alpha=1$.
	The error of the approximants in these spaces is shown in Figure \ref{fig:slid:alpha6_error}, the polynomial approximation is only accurate up to $10^{-3}$ in the middle of the interval and near the endpoints the error is large.
	The rational function approximation still achieves an accuracy of the order $10^{-10}$.
	In Figure \ref{fig:slid:alpha6_metrics} the metrics for both approaches are shown, the up-and downdating is performed accurately.
	For the metric $\textrm{err}_\textrm{f}$ for $h(x)$ we observe, up to oscillations, almost no increase for this metric. 
	These oscillations are a consequence of the approximation quality of our chosen space at each step $k$ for the function $f(x)$ on the interval and thus not a consequence of the linear algebra problem behind.
	For this metric for the polynomial approximation, it is bad and remains bad, but this again is a consequence of the approximation quality.
	The overall shape of the error of $g(x)$, $10^{-4}$ in the middle and peaks near the endpoints, remains the same during the sliding of the window.
	\begin{figure}[!ht]
		\centering
		\includegraphics{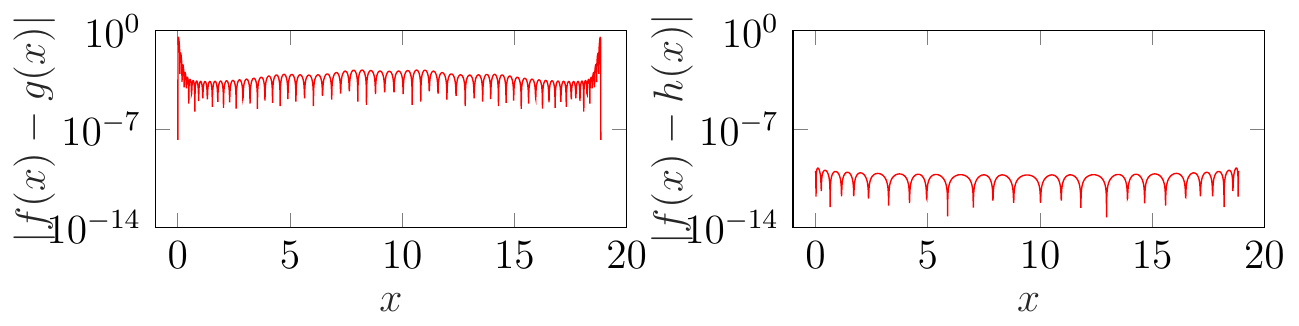}
		\caption{Error for least squares approximants $g(x)\in\mathcal{P}_{65}$ and $h\in\mathcal{R}_{25}^\Xi$ to $f(x) = (\cos^2 x +1)^{-1}$ using $m=201$ equidistant nodes in $\left[0,6\pi\right]$.}
		\label{fig:slid:alpha6_error}
	\end{figure}
	
	\begin{figure}[!ht]
		\centering
		\includegraphics{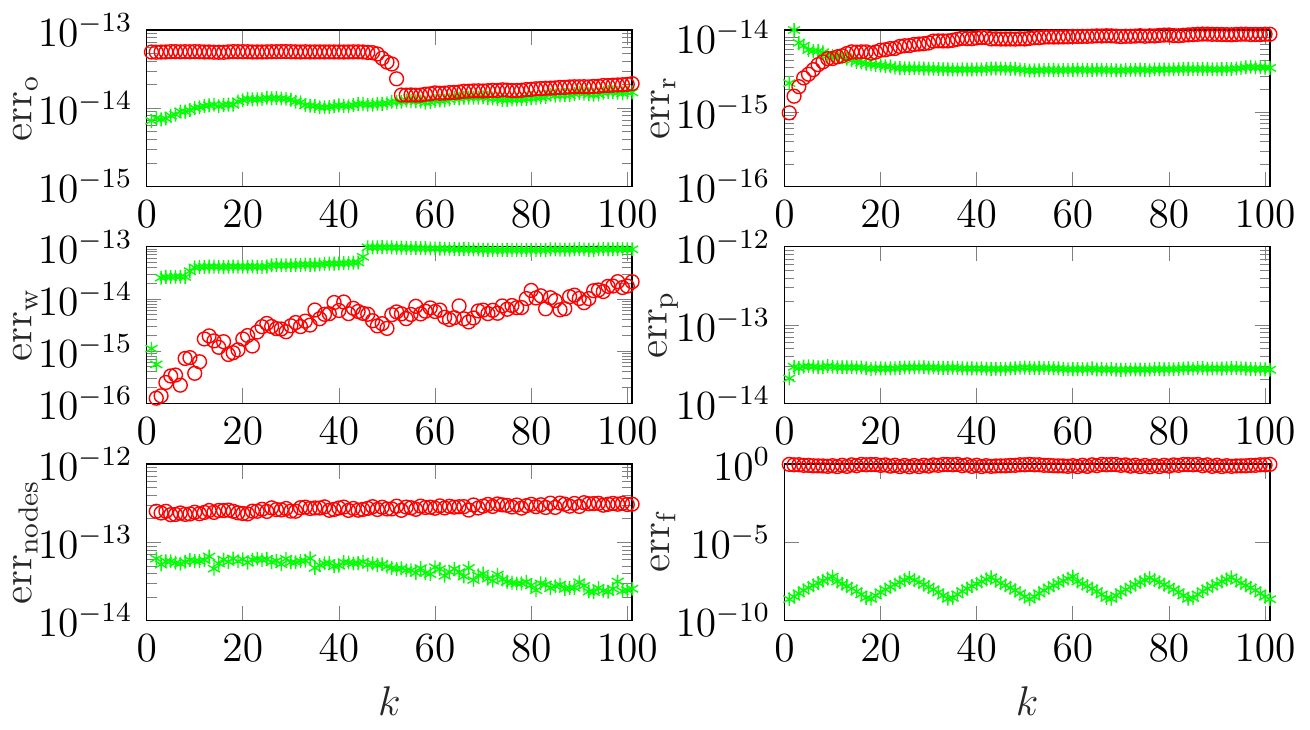}
		\caption{Metrics for sliding window of $m=201$ equidistant nodes on $\left[0,6\pi\right]$ to $m$ equidistant nodes on $\left[6\pi,12\pi\right]$. Metric $\textrm{err}_\textrm{f}$ for least squares approximants $g(x)\in\mathcal{P}_{65}$ and $h\in\mathcal{R}_{25}^\Xi$ to $f(x)=(\cos^2 x +1)^{-1}$ in this sliding window.
			For the polynomial space with the eigenvector method as downdating procedure ${\color{red} \circ}$ and the rational function space with the implicit matrix method as downdating procedure ${\color{green}\ast}$.}
		\label{fig:slid:alpha6_metrics}
	\end{figure}
	\newpage
	\section*{Acknowledgements}
	The research of the first author was partially funded by
	the Research Council KU Leuven, C1-project C14/17/073 (Numerical Linear Algebra and Polynomial Computations) and by the Fund for Scientific Research–Flanders (Belgium), EOS Project no 30468160;
	and the research of the second author by
	Charles University Research program No. PRIMUS/21/SCI/009;
	and the research of the third author partially by
	the Research Council KU Leuven (Belgium), project 
	C16/21/002 (Manifactor: Factor Analysis for Maps into Manifolds) and by the Fund for Scientific Research -- Flanders (Belgium), project G0A9923N (Low rank tensor approximation techniques for up- and downdating of massive online time series clustering);
	and the research of the first and third author is funded partially by the Fund for Scientific Research -- Flanders (Belgium), project G0B0123N (Short recurrence relations for rational Krylov and orthogonal rational functions inspired by modified moments).
	
\newpage
 \bibliographystyle{elsarticle-harv} 
\bibliography{references,TOTAL} 
	

\end{document}